\documentclass[<opyions>]{elsarticle}
\usepackage{amsfonts}
\usepackage{tipa}
\usepackage{mathrsfs}
\usepackage{amssymb}
\usepackage{latexsym,bm}
\usepackage{amsthm}
\usepackage{yhmath}
\usepackage{booktabs}
\usepackage{multirow}
\usepackage[center]{caption}
\usepackage{graphics}
\usepackage{subfig}
\usepackage{color}
\usepackage{algorithmic}
\usepackage{algorithm}
\usepackage{geometry}
\usepackage{textcomp}

\usepackage{amsmath}
\numberwithin{equation}{section}
\journal{arXiv}

\newtheorem{prop}{Proposition}[section]
\newtheorem{thm}[prop]{Theorem}

\newtheorem{lem}[prop]{Lemma}

\begin{document}
\begin{frontmatter}
\title{Analytical and numerical investigation on the tempered time-fractional operator with application to the Bloch equation and the two-layered problem}

\author[qut]{Libo~Feng}

\author[qut]{Fawang~Liu\corref{cor1}}
\cortext[cor1]{Corresponding author:~f.liu@qut.edu.au. (F.~Liu).}

\author[sut]{Vo V. Anh}

\author[SIAT]{Shanlin Qin }

\address[qut]{School of Mathematical Sciences, Queensland University of Technology, GPO Box 2434, Brisbane, QLD 4001, Australia}

\address[sut]{Faculty of Science, Engineering and Technology, Swinburne University of Technology, PO Box 218, Hawthorn, VIC
3122, Australia}

\address[SIAT]{Shenzhen Institutes of Advanced Technology, Chinese Academy of Sciences, Shenzhen, 518055, China}

\begin{abstract}
In the continuous time random walk model, the time-fractional operator usually expresses an infinite waiting time probability density. Different from that usual setting, this work considers the tempered time-fractional operator, which reflects a finite waiting time probability density. Firstly, we analyse the solution of a tempered benchmark problem, which shows a weak singularity near the initial time. The L1 scheme on graded mesh and the WSGL formula with correction terms are adapted to deal with the non-smooth solution, in which we compare these two methods systematically in terms of the convergence and consumed CPU time. Furthermore, a fast calculation for the time tempered Caputo fractional derivative is developed based on a sum-of-exponentials approximation, which significantly reduces the running time. Moreover, the tempered operator is applied to the Bloch equation in nuclear magnetic resonance and a two-layered problem with composite material exhibiting distinct memory effects, for which both the analytical (or semi-analytical) and numerical solutions are derived using transform techniques and finite difference methods. Data fitting results verify that the tempered time-fractional model is much effective to describe the MRI data. An important finding is that, compared with the fractional index, the tempered operator parameter could further accelerate the diffusion. The tempered model with two parameters $\alpha$ and $\rho$ are more flexible, which can avoid choosing a too small fractional index leading to low regularity and strong heterogeneity.

\end{abstract}
\begin{keyword}
Tempered operator \sep Tempered Bloch equation \sep Two-layered problem \sep Tempered diffusion \sep Non-smooth solution
\end{keyword}

\end{frontmatter}

\section{Introduction}

As a generalisation of the Brownian random walk, the continuous-time random walk (CTRW) incorporates a waiting time probability density function (pdf) $\psi (t)$ and a jump length pdf $\eta (x)$. It assembles the particles having different jumps ($x_{1},x_{2},\ldots ,x_{i},\ldots $) at times ($t_{1},t_{2},\ldots ,t_{i},\ldots $). Defining the jump pdf $\Phi (x,t)$,
then the jump length pdf and the waiting time pdf can be deduced respectively as
\begin{equation*}
\eta (x)=\int_{0}^{\infty }\Phi (x,t)dt,\quad \psi (t)=\int_{-\infty }^{\infty }\Phi (x,t)dx.
\end{equation*}
Considering the case that the jump length and waiting time are independent, $\Phi (x,t)$ can be written as $\Phi (x,t)=\psi (t)\eta (x)$. For a given $\psi (t)$ and $\eta (x)$, the probability of finding a particle at position $x$ and time $t$ satisfies the Montroll-Weiss (MW) equation \cite{Montroll,Metzler}, which has the following form in the Fourier-Laplace space:
\begin{equation*}
\widehat{\overline{P}}(\omega ,s)=\frac{1-\overline{\psi }(s)}{s}\frac{1}{1-\overline{\psi }(s)\hat{\eta}(\omega )},
\end{equation*}
where $\overline{\psi }(s)=\int_{0}^{\infty }e^{st}\psi (t)dt$ and $\hat{\eta}(\omega )=\int_{-\infty }^{\infty }e^{i\omega x}\eta (x)dx$. When a long-tailed waiting time pdf with the asymptotic behaviour \cite{Metzler} $\psi (t)\sim A_{\alpha }(C_{\tau }/t)^{1+\alpha }$, $0<\alpha <1$, is considered, the Caputo time-fractional operator ${_{0}^{C}D_{t}^{\alpha }}$ can be introduced in the setting. In this case, the corresponding characteristic waiting time is
\begin{equation*}
T=\int_{0}^{\infty }t\psi (t)dt=A_{\alpha }C_{\tau }^{1+\alpha
}\int_{0}^{\infty }t^{-\alpha }dt=\infty .
\end{equation*}
In the real world, since the waiting times  may not be arbitrarily long with a pure power-law distribution, the concept of tempered waiting time pdf was proposed \cite{Mantegna, Koponen}, which is expected to have broader applicability \cite{Meerschaert08}. Now consider a tempered long-tailed waiting time pdf \cite{Meerschaert15} $\psi (t)\sim A_{\alpha }(C_{\tau}/t)^{1+\alpha }e^{-\rho t}$ with Laplace transform $\overline{\psi }(s)\sim1-(C_{\tau }(s+\rho ))^{\alpha }$, then the Caputo tempered time-fractional operator ${_{0}^{C}D_{t}^{(\alpha ,\rho )}}$ can be introduced as
\begin{equation*}
{_{0}^{C}D_{t}^{(\alpha ,\rho )}}u(t)=e^{-\rho t}{_{0}^{C}D_{t}^{\alpha }}
(e^{\rho t}u(t))=\frac{e^{-\rho t}}{\Gamma (1-\alpha )}\int_{0}^{t}\frac{
d(e^{\rho s}u(s))}{ds}\frac{1}{(t-s)^{\alpha }}ds.
\end{equation*}
The corresponding characteristic waiting time turns out to be
\begin{equation*}
T=\int_{0}^{\infty }t\psi (t)dt=A_{\alpha }C_{\tau }^{1+\alpha
}\int_{0}^{\infty }t^{-\alpha }e^{-\rho t}dt=A_{\alpha }C_{\tau }^{1+\alpha }\frac{\Gamma (1-\alpha )}{\rho ^{1-\alpha }},
\end{equation*}
which is finite. Motivated by this, we investigate the Caputo tempered fractional operator in this paper.

In the past two decades, some attention has been paid to the investigation of the tempered operator. In \cite{Mantegna, Koponen}, the truncated L\'{e}vy process was introduced to eliminate the arbitrarily large flights produced by L\'{e}vy stable distributions. In \cite{Cartea07}, Cartea and del-Castillo-Negrete considered the CTRWs with exponentially truncated L\'{e}vy jump pdf, which led to a transport equation with space tempered fractional derivatives to describe the interaction between long jumps and memory in the intermediate asymptotic regime. In \cite{Meerschaert08}, Meerschaert et al. proposed a novel tempered anomalous diffusion model with exponentially tempered waiting times to capture the natural cutoff of retention times in heterogeneous systems. In \cite{Zhang12}, Zhang et al. presented a tempered fractional mobile-immobile model with a tempered stable pdf for both particle jump length and resting time able to explain the sediment transport involving super-diffusive, sub-diffusive, and regular
diffusive dynamics. In \cite{Deng16}, Wu et al. derived the tempered fractional Feynman-Kac equation to describe the functional distribution of the paths of tempered anomalous dynamics. In \cite{Boniece20}, Boniece et al. constructed two classes of second-order, non-Gaussian transient anomalous diffusion models to depict the tempered fractional L\'{e}vy process.

In addition, different numerical methods were proposed to deal with the tempered fractional models \cite{Li16,Zayernouri15,Zhang16,Guo18}, most of which are based on the spatial tempered operator. Here, we concentrate on a
numerical treatment of the tempered time-fractional model. In \cite{Hanert14}, Hanert and Piret applied a Chebyshev pseudospectral method to discretise a time-space tempered fractional diffusion equation. In \cite{Chen18}, Chen and Deng proposed some high order finite difference schemes for the time-tempered fractional Feynman-Kac equation. In \cite{Ding19}, Ding and Li developed a tempered fractional-compact difference formula to deal with a time Caputo-tempered partial differential equation. In \cite{Cao20}, Cao et
al. considered the finite element method for a tempered time fractional advection-dispersion equation.

Most existing numerical methods for the time tempered models assume a smooth solution, while the works on numerical methods dealing with non-smooth solutions are sparse. The time tempered benchmark problem shows that the solution has a weak singularity near the initial time, which motivates us to develop numerical methods to handle non-smooth solutions. To observe the behaviour of the tempered operator, we adapt the tempered operator to the Bloch equation and the two-layered problem. The main contributions of the
present work are as follows:
\begin{itemize}
\item[$\bullet $] Beginning with a tempered benchmark problem, the exact expression of the solution is derived, which incorporates the Mittag-Leffler function and an exponentially tempered factor. When $t\rightarrow 0$, the solution shows weak singularity near the initial time, for which most existing high-order numerical schemes would fail. Based on the regularity of the solution to the benchmark problem, a modified L1 scheme on graded mesh and the WSGL formula with correction terms are developed, for which a systematical comparison in terms of the convergence and consumed CPU time is presented. It finds that for the L1 scheme on graded mesh, changing the regularity index $\alpha $ does not affect the total consumed CPU time too much, which facilitates the implementation of the fast algorithm. The WSGL scheme with correction terms shows high accuracy and needs more correction terms when $\alpha $ is small, which increases the consumed CPU time and leads to an  ill-conditioned matrix. A fast calculation for the time tempered Caputo fractional derivative is developed based on a sum-of-exponentials approximation, in which the kernel function $t^{-\beta }$ is approximated under desired precision $\epsilon $ using three different quadratures. Numerical results demonstrate that this fast method can reduce the running time significantly.

\item[$\bullet$] For the tempered Bloch equation, the analytical solution is presented using Laplace transform. The numerical solution is also deduced, by which the effects of the fractional index $\alpha$ and tempered parameter $\rho$ are analysed in detail. The fractional index delays the `longitudinal' relaxation and promotes the `transversal' relaxation. In comparison, the tempered parameter accelerates the `longitudinal' relaxation and alters the asymptotic value and further advances the `transversal' relaxation. In addition, the classical monoexponential model, the time-fractional model and the tempered time-fractional model are compared to fit the MRI signal data in the human brain, of which the tempered time-fractional model performs the best.

\item[$\bullet $] For the tempered diffusion problem, the associated mean squared displacement has one more factor $e^{-\rho t}$ compared to that of the pure time-fractional diffusion problem. The analytical solution using finite Fourier transform and numerical solution with stability and convergence analysis are applied to solve the problem, which is also extended to the two-layered problem composing two different materials. An important finding is that, compared with the fractional index, the tempered parameter could further accelerate the diffusion. The tempered model with two parameters $\alpha $ and $\rho $ are more flexible, which can avoid choosing a too small fractional index leading to low regularity and strong heterogeneity.
\end{itemize}

The structure of this paper is as follows. In Section \ref{Sec2}, we consider a tempered benchmark problem, in which two classical numerical schemes are applied and the fast calculation for the time tempered Caputo fractional derivative is developed. In Section \ref{Sec3}, the tempered Bloch equation is proposed, which is solved analytically and numerically. In Section \ref{Sec4}, the tempered diffusion problem is presented, for which the analytical and numerical solutions are established. In Section \ref{Sec5}, a tempered two-layered problem is introduced. Semi-analytical solution and numerical solution are derived. Finally, some  conclusions are drawn in Section \ref{Sec6}.

\section{Benchmark problems} \label{Sec2}

In this work, we denote $C$ as a general constant independent of the grid size and may take distinct values in different contexts. Firstly, we start with the following benchmark problem:
\begin{align} \label{eq2.1}
{_0^CD_t^{(\alpha,\rho)}}u(t)=-k_0 u(t),~0<t\leq T,\quad u(0)=u_0,\quad 0<\alpha<1,\quad \rho\geq 0,
\end{align}
where $k_{0}$ and $u_{0}$ are some constants. If $\rho =0$, Problem (\ref{eq2.1}) reduces to the fractional benchmark problem \cite{Feng21}.

\subsection{Analytical solution} \label{Sec2.1}

Define the Laplace transform $\mathcal{L}\left\{ f(t)\right\} =\bar{f}(s)=\int_{0}^{\infty }e^{-st}f(t)dt$ and the inverse Laplace transform $f(t)=\mathcal{L}^{-1}\left\{ \bar{f}(s)\right\} =\frac{1}{2\pi i}\int_{c-i\infty}^{c+i\infty }e^{st}\bar{f}(s)ds$. Applying the Laplace transform to (\ref{eq2.1}) using the formula $\mathcal{L}\left\{ {_{a}^{C}D_{t}^{(\alpha ,\rho)}}u(t)\right\} =(s+\rho )^{\alpha }\bar{u}(s)-(s+\rho )^{\alpha -1}u(0)$, we have
\begin{equation*}
(s+\rho )^{\alpha }\bar{u}(s)-(s+\rho )^{\alpha -1}u(0)=-k_{0}\bar{u}(s),
\end{equation*}
which leads to
\begin{equation*}
\bar{u}(s)=\frac{(s+\rho )^{\alpha -1}}{(s+\rho )^{\alpha }+k_{0}}u_{0}.
\end{equation*}
Utilising the property $\mathcal{L}\left\{ e^{\rho t}f(t)\right\} =\bar{f}(s-\rho )$ and $\mathcal{L}\left\{ E_{\alpha }(-k_{0}t^{\alpha })\right\} =\frac{s^{\alpha -1}}{s^{\alpha }+k_{0}}$, we can derive
\begin{equation*}
u(t)=u_{0}e^{-\rho t}E_{\alpha }(-k_{0}t^{\alpha }),
\end{equation*}
where $E_{\alpha }(z)=\sum_{n=0}^{\infty }\frac{z^{n}}{\Gamma (n\alpha +1)}$ is the Mittag-Leffler function. When $t\rightarrow 0$, using the Taylor series $e^{-\rho t}=\sum_{j=0}^{\infty }\frac{(-\rho t)^{j}}{j!}$, we obtain
\begin{align} \label{eq2.5}
u(t)=u_0e^{-\rho t}E_{\alpha}(-k_0 t^\alpha)=u_0\sum_{j=0}^{\infty}\frac{(-\rho t)^j}{j!}\sum_{n=0}^{\infty}\frac{(-k_0t^\alpha)^n}{\Gamma(n\alpha+1)}=u_0\sum_{j=0}^{\infty}\sum_{n=0}^{\infty}C_{nj}t^{j+n\alpha}.
\end{align}
We can see that, similar to the fractional benchmark problem, the solution shows a weak singularity near the initial time $t=0$ since $\frac{du(t)}{dt}$ blows up as $t\rightarrow 0$. We will extend two classical methods to deal with the weak singularity problem: the L1 scheme on graded mesh and the WSGL scheme with correction terms.

\subsection{L1 scheme on graded mesh} \label{Sec2.2}

Let $N$ be a positive integer. Denote $t_{n}=T\left( \frac{n}{N}\right) ^{r}$, $r\geq 1$, $n=0,1,2,\ldots ,N$, $\tau _{n}=t_{n}-t_{n-1}$, $n=1,2,\ldots ,N$. When $r=1$, the mesh recovers the uniform mesh. Recall the L1 formula on
graded mesh \cite{Stynes17} to approximate the Caputo derivative at $t=t_{n}$:
\begin{equation*}
{_{0}^{C}D_{t}^{\alpha }}u(t_{n})=\frac{1}{\Gamma (2-\alpha )}\sum_{k=0}^{n-1}\frac{u(t_{k+1})-u(t_{k})}{\tau _{k+1}}\left[
(t_{n}-t_{k})^{1-\alpha }-(t_{n}-t_{k+1})^{1-\alpha }\right] +R_{1}^{n}.
\end{equation*}
According to the definition of the Caputo tempered fractional derivative, it is straightforward to derive
\begin{align*}
& {_{0}^{C}D_{t}^{(\alpha ,\rho )}}u(t_{n})=e^{-\rho t_{n}}{_{0}^{C}D_{t}^{\alpha }}(e^{\rho t_{n}}u(t_{n}))=\frac{\tau _{n}^{-\alpha }}{\Gamma (2-\alpha )}u(t_{n})-\frac{\tau _{n}^{-\alpha }}{\Gamma (2-\alpha )}
e^{-\rho \tau _{n}}u(t_{n-1})   \\
+& \frac{1}{\Gamma (2-\alpha )}\sum_{k=0}^{n-2}\frac{e^{-\rho (t_{n}-t_{k+1})}u(t_{k+1})-e^{-\rho (t_{n}-t_{k})}u(t_{k})}{\tau _{k+1}}\left[ (t_{n}-t_{k})^{1-\alpha }-(t_{n}-t_{k+1})^{1-\alpha }\right]
+R_{1}^{n}  \notag \\
:=& {_{0}^{C}\mathbb{D}_{t}^{(\alpha ,\rho )}}u(t_{n})+R_{1}^{n}.
\end{align*}
For the bound of the error term $R_{1}^{n}$, we have the following lemma.

\begin{lem}\label{lem2.1}
Suppose $|\frac{du}{dt}|\leq C(1+t^{\alpha -l})$ for $l=0,1,2$ and $0<\alpha <1$, then the truncation error of the L1 scheme satisfies
\begin{align} \label{eq2.8}
|R_1^n|=\left|{_0^CD_t^{(\alpha,\rho)}}u(t_n)- {_0^C\mathbb{D}_t^{(\alpha,\rho)}}u(t_n)\right|\leq C n^{-\min\{2-\alpha,r\alpha\}}.
\end{align}
\end{lem}

\begin{proof}
Using the definition of the derivative and integration by parts, we have
\begin{align*}
&{_0^CD_t^{(\alpha,\rho)}}u(t_n)- {_0^C\mathbb{D}_t^{(\alpha,\rho)}}u(t_n)=\frac{e^{-\rho t_n}}{\Gamma(1-\alpha)}\sum_{k=0}^{n-1}\int_{t_k}^{t_{k+1}}\frac{ \left(e^{\rho s}[u(s)-u_h(s)]\right)'}{(t_n-s)^{\alpha}}ds\\
=&\frac{-\alpha e^{-\rho t_n}}{\Gamma(1-\alpha)}\sum_{k=0}^{n-1}\int_{t_k}^{t_{k+1}}(t_n-s)^{-(\alpha+1)} e^{\rho s}[u(s)-u_h(s)]ds=\sum_{k=0}^{n-1}T_{n,k},
\end{align*}
where $u_h(t)$ is a linear interpolation function to approximate $u(t)$ in $[t_k,t_{k+1}]$ and
\begin{align*}
T_{n,k}=\frac{-\alpha e^{-\rho t_n}}{\Gamma(1-\alpha)}\int_{t_k}^{t_{k+1}}(t_n-s)^{-(\alpha+1)} e^{\rho s}[u(s)-u_h(s)]ds.
\end{align*}
As $u(s)-u_h(s)=\frac{1}{2}u''(\xi_k)(s-t_{k-1})(t_k-s)$, $\xi_k\in[t_{k-1},t_k]$, we have
\begin{align*}
|T_{n,k}|\leq C\tau_{k+1}^2\max\limits_{t\in[t_k,t_{k+1}]}|u''(t)| \int_{t_k}^{t_{k+1}}(t_n-s)^{-(\alpha+1)}ds.
\end{align*}
Similar to the proof in Lemma 5.2 in \cite{Stynes17}, when $1\leq k\leq n-1$, we have
\begin{equation*}
\sum_{k=1}^{\lceil n/2\rceil-1}|T_{n,k}|\leq
\left\{\begin{array}{ll}
Cn^{-r(\alpha+1)},& {\rm{if}}~ r(\alpha+1)<2,\\
Cn^{-2}\ln n,& {\rm{if}}~ r(\alpha+1)=2,\\
Cn^{-2},& {\rm{if}}~ r(\alpha+1)>2,
\end{array}\right.
\quad
\text{and}\quad \sum_{k=\lceil n/2\rceil}^{n-2}|T_{n,k}|\leq Cn^{-(2-\alpha)}.
\end{equation*}
Now we consider the bound of $T_{n,0}$. If $n=1$, then
\begin{align*}
T_{1,0}=\frac{\tau_1^{-\alpha}}{\Gamma(2-\alpha)}[F(t_1)-F(t_0)]
-\frac{1}{\Gamma(1-\alpha)}\int_0^{t_1}(t_1-s)^{-\alpha}F'(s)ds,
\end{align*}
where $F(t)=e^{\rho t}u(t)$. As $F'(t)=e^{\rho t}(\rho u(t)+u'(t))$, then $|F'(t)|\leq Ct^{\alpha-1}$, $t\in(0,t_1)$. Similar to the proof in \cite{Stynes17}, we can obtain $|T_{1,0}|\leq C$. For $n>1$, we have
\begin{align*}
&|T_{n,0}|=\left|\frac{\alpha e^{-\rho t_n}}{\Gamma(1-\alpha)}\int_{t_0}^{t_{1}}(t_n-s)^{-(\alpha+1)} e^{\rho s}[u(s)-u_h(s)]ds\right|\leq C\int_{t_0}^{t_{1}}(t_n-s)^{-(\alpha+1)}|u(s)-u_h(s)| ds\\
=&C\int_{t_0}^{t_{1}}(t_n-s)^{-(\alpha+1)}\left|\int_0^s(u(\theta)-u_h(\theta))'d\theta\right|ds\leq C\int_{t_0}^{t_{1}}(t_n-s)^{-(\alpha+1)}\left(\int_0^s|u'(\theta)|d\theta+
\int_0^st_1^{-1}\int_0^{t_1}|u'(\eta)|d\eta d\theta\right)ds\\
\leq &C\int_{t_0}^{t_{1}}(t_n-s)^{-(\alpha+1)}(s^\alpha+st_1^{\alpha-1})ds\leq Ct_1^{\alpha}\int_{t_0}^{t_{1}}(t_n-s)^{-(\alpha+1)}ds\leq C \left(\frac{t_n-t_1}{t_1} \right)^{-\alpha}\leq Cn^{-r\alpha}.
\end{align*}
Finally, we consider the term $T_{n,n-1}$
\begin{align*}
|T_{n,n-1}|&\leq C\tau_{n}^2\max\limits_{t\in[t_{n-1},t_{n}]}|u''(t)| \int_{t_{n-1}}^{t_{n}}(t_n-s)^{-(\alpha+1)}ds\leq C\tau_{n}^2 t_{n-1}^{\alpha-2}\tau_{n}^{-\alpha}\leq  C\tau_{n}^{2-\alpha} t_{n-1}^{\alpha-2}
\leq Cn^{-(2-\alpha)}.
\end{align*}
Combining these error bounds, we can derive (\ref{eq2.8}).
\end{proof}
Denote $u^{n}$ as the numerical approximation to $u(t_{n})$. We can derive the numerical scheme of (\ref{eq2.1}) as
\begin{equation*}
{_{0}^{C}\mathbb{D}_{t}^{(\alpha ,\rho )}}u^{n}=-k_{0}u^{n}.
\end{equation*}
We give an example to illustrate the effectiveness of the numerical scheme by choosing $k_{0}=2$, $\rho =0.5$, $u_{0}=1$ and $T=1$. The error and convergence order of the scheme with varying $N$ and $\alpha $ are presented in Table \ref{tab1}. It can be seen that the L1 scheme fails for a uniform mesh ($r=1$) while the highest possible convergence order $2-\alpha $ is obtained for the optimal graded mesh $r=\frac{2-\alpha }{\alpha }$, which means the L1 scheme is effective as well to approximate the time tempered
operator. As suggested by \cite{Stynes17}, it is better to choose $r=\frac{2(2-\alpha )}{\alpha }$ when the fractional index $\alpha $ is unknown.

\begin{table}[H] 
\caption{The error and convergence order of the L1 scheme on graded mesh for the tempered ODE (\ref{eq2.1}) at $t=1$, where the parameters are $k_0=2$, $\rho=0.5$, $u_0=1$.}
\label{tab1}
\centering
\begin{tabular}{c|  c c| c c|cc}
\toprule
\multirow{2}{*}{\centering $N$ ($\alpha=0.8$)}
 &\multicolumn{2}{c|}{$r=1$} &\multicolumn{2}{c|}{$ r=\frac{2-\alpha}{\alpha}$} &\multicolumn{2}{c}{$ r=\frac{2(2-\alpha)}{\alpha}$}\\
\cmidrule{2-7}
& Error & Order&  Error   &Order&  Error   &Order  \\
\midrule
$160$  & 6.0205E-03     &        &  1.5928E-03      &        & 9.5021E-04      &       \\
$320$  & 3.4550E-03     &  0.80  &  7.3284E-04      &  1.12  & 4.1541E-04      & 1.19  \\
$640$  & 1.9798E-03     &  0.80  &  3.3371E-04      &  1.13  & 1.8123E-04      & 1.20  \\
$1280$ & 1.1365E-03     &  0.80  &  1.5075E-04      &  1.15  & 7.8981E-05      & 1.20  \\
$2560$ & 6.5228E-04     &  0.80  &  6.7666E-05      &  1.16  & 3.4401E-05      & 1.20  \\
$5120$ & 3.7444E-04     &  0.80  &  3.0218E-05      &  1.16  & 1.4979E-05      & 1.20  \\
\midrule
 \multirow{2}{*}{\centering $N$ ($\alpha=0.4$)}
 &\multicolumn{2}{c|}{$r=1$} &\multicolumn{2}{c|}{$ r=\frac{2-\alpha}{\alpha}$} &\multicolumn{2}{c}{$ r=\frac{2(2-\alpha)}{\alpha}$}\\
\cmidrule{2-7}
& Error  &  Order  &  Error   &  Order  &  Error   &  Order  \\
\midrule
$160$  & 4.5385E-02     &        &  3.4393E-04      &        & 2.3495E-04      &       \\
$320$  & 3.6943E-02     &  0.30  &  1.1842E-04      &  1.54  & 7.9816E-05      & 1.56  \\
$640$  & 2.9574E-02     &  0.32  &  4.0418E-05      &  1.55  & 2.6902E-05      & 1.57  \\
$1280$ & 2.3372E-02     &  0.34  &  1.3712E-05      &  1.56  & 8.9942E-06      & 1.58  \\
$2560$ & 1.8287E-02     &  0.35  &  4.6283E-06      &  1.57  & 2.9968E-06      & 1.59  \\
$5120$ & 1.4201E-02     &  0.36  &  1.5557E-06      &  1.57  & 9.8013E-07      & 1.61  \\
\bottomrule
\end{tabular}
\end{table}

\subsection{The WSGL formula with correction terms} \label{Sec2.3}

Generally, the analytical solution of the time-fractional differential equation contains the term $t^{\sigma _{m}}$, which exhibits low regularity when ${\sigma _{m}}$ is small. The weighted shifted Gr\"{u}nwald-Letnikov (WSGL) formula with correction terms is a useful method to deal with the solution with low regularity term $t^{\sigma _{m}}$, in which the correction term can be exact or highly accurate to approximate the term $t^{\sigma _{m}}$ \cite{Lubich,Zeng17,Liu2019}. Even when the regularity indices in the
correction terms do not match the singularity of the analytical solution exactly, the accuracy of the WSGL formula can still be improved significantly \cite{Zeng17}. As shown in (\ref{eq2.5}), the solution incorporates a similar term $t^{\sigma _{m}}$, therefore the WSGL formula could be used to handle the problem. Here we use a uniform grid in the temporal direction. Denote $t_{n}=n\tau $, $n=0,1,2,\ldots ,N$, where $\tau =\frac{T}{N}$ is the uniform temporal step and $N\in \mathbb{Z}^+$. Introduce the Riemann-Liouville fractional derivative
\begin{equation*}
{_{0}^{R}D_{t}^{\alpha }}v(t)=\frac{1}{\Gamma (1-\alpha )}\frac{d}{dt}
\int_{0}^{t}\frac{v(\xi )}{(t-\xi )^{\alpha }}d\xi ,\quad 0<\alpha <1.
\end{equation*}
Then the WSGL formula to discretise the Riemann-Liouville time-fractional derivative is
\begin{align} \label{eq2.11}
{_0^RD_t^{\alpha}}u(t_n)=\tau^{-\alpha}\sum_{k=0}^n\omega_{n-k}^{(\alpha)}u(t_k)
+\tau^{-\alpha}\sum_{k=1}^m W_k^{(n,\alpha)}u(t_k),
\end{align}
where $\omega _{0}^{(\alpha )}=\frac{2+\alpha }{2}g_{0}^{(\alpha )}$ $\omega_{k}^{(\alpha )}=\frac{2+\alpha }{2}g_{k}^{(\alpha )}-\frac{\alpha }{2}g_{k-1}^{(\alpha )}$, $g_{k}^{(\alpha )}=(-1)^{k}{\binom{\alpha }{k}}$, and the starting weights $W_{k}^{(n,\alpha )}$ are chosen such that (\ref{eq2.11}) is exact for $v(t)=t^{\sigma _{m}}$ ($m=0,1,2,\ldots $), which leads to
the system:
\begin{align} \label{eq2.12}
\sum_{k=1}^{m}W_k^{(n,\alpha)}k^{\sigma_m}=\frac{\Gamma(\sigma_m+1)}{\Gamma(\sigma_m+1-\alpha)}n^{\sigma_m-\alpha}
-\sum_{k=0}^{n}\omega_{n-k}^{(\alpha)}k^{\sigma_m}.
\end{align}
We apply the formula to the Caputo tempered operator at $t=t_{n}$ to obtain
\begin{align*}
{_{0}^{C}D_{t}^{(\alpha ,\rho )}}u(t_{n})& =e^{-\rho t_{n}}{_{0}^{R}D_{t}^{\alpha }}\left( F(t_{n})-F(t_{0})\right)  \\
& =e^{-\rho t_{n}}\tau ^{-\alpha }\sum_{k=0}^{n}\omega _{n-k}^{(\alpha)}\left( F(t_{k})-F(t_{0})\right) +e^{-\rho t_{n}}\tau ^{-\alpha}\sum_{k=1}^{m}W_{k}^{(n,\alpha )}\left( F(t_{k})-F(t_{0})\right) +R_{2}^{n}\\
& =\tau ^{-\alpha }\sum_{k=0}^{n}\omega _{n-k}^{(\alpha )}e^{-\rho
t_{n-k}}u(t_{k})-e^{-\rho t_{n}}\tau ^{-\alpha }\sum_{k=0}^{n}\omega
_{n-k}^{(\alpha )}u(t_{0}) \\
& +\tau ^{-\alpha }\sum_{k=1}^{m}W_{k}^{(n,\alpha )}e^{-\rho
t_{n-k}}u(t_{k})-e^{-\rho t_{n}}\tau ^{-\alpha
}\sum_{k=1}^{m}W_{k}^{(n,\alpha )}u(t_{0})+R_{2}^{n} \\
& :={_{0}^{C}\mathfrak{D}_{t}^{(\alpha ,\rho )}}u(t_{n})+R_{2}^{n},
\end{align*}
where $F(t)=e^{\rho t}u(t)$ and $|R_{2}^{n}|\leq C\tau ^{2}t_{n}^{\sigma_{m+1}-2-\alpha }$. To guarantee the WSGL formula has a global second-order convergence, we need $\sigma _{m+1}\geq 2+\alpha $. For more details about choosing the correction terms, one can refer to \cite{Zeng17}. For our problem, at $t=t_{n}$, we have
\begin{equation*}
{_{0}^{C}\mathfrak{D}_{t}^{(\alpha ,\rho )}}u^{n}=-k_{0}u^{n}.
\end{equation*}
To verify the WSGL formula, we also give an example using the same parameters $k_{0}=2$, $\rho =0.5$, $u_{0}=1$ and $T=1$, to which the numerical results are given in Table \ref{tab2}. We can observe that without the correction terms ($m=0$), the WSGL formula only exhibits convergence rate $O(\tau ^{\alpha })$. While adding the correction terms into the WSGL formula, the second-order convergence is obtained for the case $\alpha =0.8$ with $m=2$. Although the optimal second-order convergence is not reached for
the case $\alpha =0.4$ with $m=4$, the accuracy has been improved significantly. It reveals that the WSGL formula is also feasible to tackle the time-fractional tempered problem.
\begin{table}[] 
\caption{The error and convergence order of the WSGL scheme for the tempered ODE (\ref{eq2.1}) at $t=1$, where the parameters are $k_0=2$, $\rho=0.5$, $u_0=1$.}
\label{tab2}
\centering
\begin{tabular}{c|  c c| c c}
\toprule
\multirow{2}{*}{\centering $N$ ($\alpha=0.8$)}
 &\multicolumn{2}{c|}{$m=0$} &\multicolumn{2}{c}{$m=2$} \\
\cmidrule{2-5}
& Error & Order&  Error   &Order    \\
\midrule
$160$  & 1.2426E-02     &        &  1.4662E-05      &            \\
$320$  & 7.3015E-03     &  0.77  &  3.8910E-06      &  1.91    \\
$640$  & 4.2476E-03     &  0.78  &  1.0616E-06      &  1.87    \\
$1280$ & 2.4573E-03     &  0.79  &  2.8079E-07      &  1.92    \\
$2560$ & 1.4171E-03     &  0.79  &  7.2789E-08      &  1.95    \\
$5120$ & 8.1580E-04     &  0.80  &  1.8625E-08      &  1.97   \\
\midrule
 \multirow{2}{*}{\centering $N$ ($\alpha=0.4$)}
  &\multicolumn{2}{c|}{$m=0$} &\multicolumn{2}{c}{$m=4$} \\
\cmidrule{2-5}
& Error  &  Order  &  Error   &  Order     \\
\midrule
$160$  & 5.5856E-02     &        &  3.1630E-05      &           \\
$320$  & 4.5653E-02     &  0.29  &  1.0970E-05      &  1.53    \\
$640$  & 3.6675E-02     &  0.32  &  3.5479E-06      &  1.63    \\
$1280$ & 2.9069E-02     &  0.34  &  1.0843E-06      &  1.71    \\
$2560$ & 2.2800E-02     &  0.35  &  3.1673E-07      &  1.78    \\
$5120$ & 1.7739E-02     &  0.36  &  8.9282E-08      &  1.83    \\
\bottomrule
\end{tabular}
\end{table}

\subsection{Comparison of the two methods} \label{Sec2.4}

In this section, we compare the L1 scheme on graded mesh with the WSGL scheme. To avoid the error caused by dsicretising the Mittag-Leffler function, we consider the following example:
\begin{align} \label{eq2.14}
{_0^CD_t^{(\alpha,\rho)}}u(t)=f(t),~ 0<t\leq T, \quad u(0)=u_0,
\end{align}
with an exact solution $u(t)=u_{0}e^{-\rho t}\sum\limits_{k=0}^{8}t^{k\alpha }$. In the calculation, we choose the parameters $\rho =0.5$, $u_{0}=1$ and $T=1$. The related numerical results are shown in Tables \ref{tab3} and \ref{tab4}, respectively. From Table \ref{tab3}, we can see that the optimal convergence order $2-\alpha $ is obtained using mesh grading $r=\frac{2(2-\alpha )}{\alpha }$. An important property is that the CPU time for the cases $\alpha =0.4$ and $\alpha =0.8$ is almost the same, which means the
fractional index $\alpha $ does not affect the CPU time too much for a fixed $N$. This will facilitate the implementation of the fast calculation of the tempered operator in the subsequent section. Compared to the L1 scheme, the WSGL scheme is of high order, which shows better accuracy and convergence. However, when $\alpha $ is small, we need to add more correction terms, which increases the CPU time. In addition, the system (\ref{eq2.12}) leads to an exponential Vandermonde type matrix. When the number of the correction terms is large, the matrix will be ill-conditioned, which may cause big roundoff errors \cite{Zeng17}. We can conclude that both methods have merits and drawbacks.
\begin{table}[] 
\caption{The error and convergence order of the L1 scheme for the example (\ref{eq2.14}) at $t=1$, where the parameters are  $\rho=0.5$, $u_0=1$, $r=\frac{2(2-\alpha)}{\alpha}$.}
\label{tab3}
\centering
\begin{tabular}{c| c c c| cc c}
\toprule
\multirow{3}{*}{\centering $N$ }
 &\multicolumn{3}{c|}{$\alpha=0.4$} &\multicolumn{3}{c}{$\alpha=0.8$} \\
\cmidrule{2-7}
         &  Error         & Order & CPU time(s) & Error            & Order & CPU time(s)   \\
\midrule
$640$    & 1.1327E-03     &       & 0.12        &  1.0984E-02      &       &  0.11  \\
$1280$   & 3.8563E-04     &  1.55 & 0.36        &  4.8006E-03      &  1.19 &  0.36  \\
$2560$   & 1.3022E-04     &  1.57 & 1.27        &  2.0947E-03      &  1.20 &  1.21  \\
$5120$   & 4.3745E-05     &  1.57 & 4.75        &  9.1312E-04      &  1.20 &  4.64  \\
$10240$  & 1.4672E-05     &  1.58 & 18.13       &  3.9780E-04      &  1.20 &  17.87 \\
$20480$  & 4.9366E-06     &  1.57 & 70.42       &  1.7324E-04      &  1.20 &  69.51 \\
\bottomrule
\end{tabular}
\end{table}

\begin{table}[] 
\caption{The error and convergence order of the WSGL scheme for the example (\ref{eq2.14}) at $t=1$, where the parameters are  $\rho=0.5$, $u_0=1$.}
\label{tab4}
\centering
\begin{tabular}{c| c c c| cc c}
\toprule
\multirow{3}{*}{\centering $N$ }
 &\multicolumn{3}{c|}{$\alpha=0.4~(m=4)$} &\multicolumn{3}{c}{$\alpha=0.8~(m=2)$} \\
\cmidrule{2-7}
         &  Error         & Order & CPU time(s) & Error            & Order & CPU time(s)   \\
\midrule
$640$    & 2.5706E-06     &       & 0.18        &  3.6710E-05      &       &  0.12  \\
$1280$   & 6.5015E-07     &  1.98 & 0.41        &  9.1970E-06      &  2.00 &  0.26  \\
$2560$   & 1.6377E-07     &  1.99 & 1.03        &  2.3018E-06      &  2.00 &  0.68  \\
$5120$   & 4.1153E-08     &  1.99 & 3.27        &  5.7579E-07      &  2.00 &  2.06  \\
$10240$  & 1.0325E-08     &  1.99 & 11.24       &  1.4399E-07      &  2.00 &  6.73 \\
$20480$  & 2.5878E-09     &  2.00 & 39.57       &  3.6004E-08      &  2.00 &  24.18 \\
\bottomrule
\end{tabular}
\end{table}

\subsection{Fast calculation for the time tempered Caputo fractional derivative}\label{Sec2.5}

In this part, based on the L1 scheme, we consider a fast evaluation for the time tempered Caputo fractional derivative, in which the derivative is to be splitted into two different parts: the local part and the history part. A fast approximation will be applied to the history part. We now give a detailed implementation. At $t=t_{n}$, we have
\begin{align*}
&{^C_0D_t^{(\alpha,\rho)}}u(t)\Big|_{t=t_n}=e^{-\rho t}\,{^C_0D^{\alpha}_t}(e^{\rho t}u(t)\Big|_{t=t_n}=\frac{e^{-\rho t_n}}{\Gamma(1-\alpha)} \int_0^{t_n}\frac{\left( e^{\rho s}u(s)\right)' }{(t_n-s)^\alpha} ds\nonumber\\ 
=&\frac{e^{-\rho t_n}}{\Gamma(1-\alpha)} \int_0^{t_{n-1}}\frac{\left( e^{\rho s}u(s)\right)' }{(t_n-s)^\alpha} ds
+\frac{e^{-\rho t_n}}{\Gamma(1-\alpha)} \int_{t_{n-1}}^{t_n}\frac{\left( e^{\rho s}u(s)\right)' }{(t_n-s)^\alpha} ds:=C_h(t_n)+C_l(t_n),
\end{align*}
where $C_{h}(t_{n})$ is the history part and $C_{l}(t_{n})$ is the local part. For the local part $C_{l}(t_{n})$, we use the backward difference scheme to approximate the first-order derivative, which gives
\begin{align*}
C_l(t_n)\approx \frac{e^{-\rho t_n}}{\Gamma(1-\alpha)}\frac{e^{\rho t_n}u(t_n)-e^{\rho t_{n-1}}u(t_{n-1})}{\tau_n}
\int_{t_{n-1}}^{t_n}\frac{1 }{(t_n-s)^\alpha} ds=\frac{u(t_n)-e^{-\rho \tau_{n}}u(t_{n-1})}{\tau_n^\alpha\Gamma(2-\alpha)}.
\end{align*}
With regards to the history part $C_{h}(t_{n})$, using integration by parts, we obtain
\begin{align}\label{eq2.17}
C_h(t_n)=\frac{e^{-\rho t_n}}{\Gamma(1-\alpha)} \int_0^{t_{n-1}}\frac{\left( e^{\rho s}u(s)\right)' }{(t_n-s)^\alpha} ds=
\frac{e^{-\rho t_n}}{\Gamma(1-\alpha)} \left[ \frac{e^{\rho t_{n-1}}u(t_{n-1})}{\tau_n^\alpha}-\frac{u(t_{0})}{t_n^\alpha}-\alpha \int_0^{t_{n-1}}\frac{e^{\rho s}u(s) }{(t_n-s)^{1+\alpha}} ds\right].
\end{align}
Next, we approximate the kernel $t^{-\beta }$ $(0<\beta <2)$ via a sum-of-exponentials (SOE) approximation. For any $\beta >0$, the kernel $t^{-\beta }$ can be written in integral form, i.e.,
\begin{align*}
t^{-\beta}=\frac{1}{\Gamma(\beta)}\int_0^{\infty} e^{-ts}s^{\beta-1}ds.
\end{align*}
To approximate this integral $(0<\sigma \leq t)$ with desired precision $\varepsilon $, the integral interval will be split into three parts $[0,\infty )=[0,a]\cup \lbrack a,p]\cup \lbrack p,\infty )$, where $p$ is
chosen such that $p=\mathcal{O}\left( \log \left( \frac{1}{\varepsilon
\sigma }\right) /\sigma \right) $. Then the integral on $[0,a]$ and $[a,p]$
can be approximated using Gauss-Jacobi quadrature and Gauss-Legendre
quadrature, respectively. Therefore, we have the following lemma \cite{Jiang17}.

\begin{lem}\label{lem2.2}
Let $0<\sigma\leq t\leq T$ ($\sigma\leq 1$ and $T\geq 1$), let $\varepsilon>0$ be the desired precision, $n_0=\mathcal{O}\left( \log\frac{1}{\varepsilon} \right)$, $M=\mathcal{O}\left( \log T\right)$, and let $\widetilde{N}=\mathcal{O} \left( \log\log\frac{1}{\varepsilon}+\log\frac{1}{\sigma}\right)$. Furthermore, let $s_{o,1},\ldots,s_{o,n_0}$ and $\omega_{o,1},\ldots,\omega_{o,n_0}$ be the nodes and weights for the $n_0$-point Gauss-Jacobi quadrature on the interval $[0,2^{-M}]$, let $s_{j,1},\ldots,s_{j,n_s}$ and $\omega_{j,1},\ldots,\omega_{j,n_s}$ be the nodes and weights for $n_s$-point Gauss-Legendre quadrature on small intervals $[2^j,2^{j+1}]$, $j=-M,\ldots, -1$, where $n_s=\mathcal{O}\left( \log\frac{1}{\varepsilon} \right)$, and let $s_{j,1},\ldots,s_{j,n_l}$ and $\omega_{j,1},\ldots,\omega_{j,n_l}$ be the nodes and weights for $n_l$-point Gauss-Legendre quadrature on large intervals $[2^j,2^{j+1}]$, $j=0,\ldots, \widetilde{N}$, where $n_l=\mathcal{O}\left( \log\frac{1}{\varepsilon} +\log\frac{1}{\sigma}\right)$, Then for $t\in(\sigma,T]$ and $\beta\in(0,2)$, we have
\begin{align*} 
 \left|t^{-\beta} -\left(\sum_{k=1}^{n_0}e^{-s_{o,k}t}\omega_{o,k}+\sum_{j=-M}^{-1}\sum_{k=1}^{n_s}e^{-s_{j,k}t}s^{\beta-1}_{j,k}\omega_{j,k}
+ \sum_{j=0}^{\widetilde{N}}\sum_{k=1}^{n_l}e^{-s_{j,k}t}s^{\beta-1}_{j,k}\omega_{j,k}\right)\right|\leq \varepsilon.
\end{align*}
\end{lem}
From the lemma, we can conclude that for $t\in[\sigma,T]$, $\sigma=\min\limits_{1\leq n\leq N}\{\tau_n \}$, there exist some positive real numbers $s_i$ and $\omega_i$ $(i=1,2,\ldots,N_{\rm{exp}})$ such that
\begin{align*} 
 \left|t^{-\beta}- \sum_{i=1}^{N_{\rm{exp}}}\omega_ie^{-s_{i}t}\right|\leq \varepsilon,\quad t\in[\sigma,T],\quad \beta\in(0,2),
\end{align*}
where
\begin{align*}
N_{\rm{exp}}=\mathcal{O}\left( \log\frac{1}{\varepsilon}\left(\log\log\frac{1}{\varepsilon}+\log\frac{T}{\sigma}\right)
+\log\frac{1}{\sigma}\left(\log\log\frac{1}{\varepsilon}+\log\frac{1}{\sigma}\right)\right).
\end{align*}
Replacing the kernel $t^{-1-\alpha}$ in the history part (\ref{eq2.17}) with the SOE approximation, we obtain
\begin{align*} 
\alpha \int_0^{t_{n-1}}\frac{e^{\rho s}u(s) }{(t_n-s)^{1+\alpha}} ds=\alpha \int_0^{t_{n-1}} \sum_{i=1}^{N_{exp}}\omega_ie^{-s_{i}(t_n-s)}e^{\rho s}u(s)  ds=\alpha\sum_{i=1}^{N_{exp}}\omega_i e^{\rho t_n}\mathcal{U}_{his,i}(t_n),
\end{align*}
where $\mathcal{U}_{his,i}(t_n)=\int_0^{t_{n-1}}e^{-(\rho+s_i)(t_n-s)}u(s) ds$. For the term $\mathcal{U}_{his,i}(t_n)$, we have the following recurrence relation
\begin{align*}
\mathcal{U}_{his,i}(t_n)=e^{-(\rho+s_i)\tau_n}\mathcal{U}_{his,i}(t_{n-1})+\int_{t_{n-2}}^{t_{n-1}}e^{-(\rho+s_i)(t_n-s)}u(s)  ds.
\end{align*}
Using a linear interpolation function to approximate $u(t)$ on the interval $[t_{n-2},t_{n-1}]$, we can derive
\begin{align*}
\int_{t_{n-2}}^{t_{n-1}}e^{-(\rho+s_i)(t_n-s)}u(s) ds
\approx& \frac{e^{-(\rho+s_i)\tau_n}}{(\rho+s_i)^2\tau_{n-1}}
\bigg[ (e^{-(\rho+s_i)\tau_{n-1}}-1+(\rho+s_i)\tau_{n-1})u(t_{n-1})\\
+&(1-e^{-(\rho+s_i)\tau_{n-1}}-e^{-(\rho+s_i)\tau_{n-1}}(\rho+s_i)\tau_{n-1})u(t_{n-2}) \bigg].
\end{align*}
Finally, the fast approximation of ${^C_0D_t^{(\alpha,\rho)}}u(t_n)$, $n\geq 2$ can be written as
\begin{align}
{^{FC}_0\mathbb{D}_t^{(\alpha,\rho)}}u(t_n)&:=\frac{u(t_n)-e^{-\rho \tau_{n}}u(t_{n-1})}{\tau_n^\alpha\Gamma(2-\alpha)}+
\frac{e^{-\rho t_n}}{\Gamma(1-\alpha)} \left[ \frac{e^{\rho t_{n-1}}u(t_{n-1})}{\tau_n^\alpha}-\frac{u(t_{0})}{t_n^\alpha}-\alpha\sum_{i=1}^{N_{exp}}\omega_i e^{\rho t_n}\mathcal{U}_{his,i}(t_n)\right] \nonumber\\\label{eq2.22}
&=\frac{u(t_n)}{\tau_n^\alpha\Gamma(2-\alpha)}-\frac{\alpha e^{-\rho \tau_{n}}u(t_{n-1})}{\tau_n^\alpha\Gamma(2-\alpha)}
-\frac{e^{-\rho t_{n}}u(t_{0})}{\Gamma(1-\alpha)t_n^\alpha}-\frac{\alpha}{\Gamma(1-\alpha)}\sum_{i=1}^{N_{exp}}\omega_i \mathcal{U}_{his,i}(t_n).
\end{align}
When $n=1$, we have
\begin{align} \label{eq2.23}
{^{FC}_0\mathbb{D}_t^{(\alpha,\rho)}}u(t_{1}):=\frac{u(t_{1})-e^{-\rho \tau_{1}}u(t_{0})}{\tau_1^\alpha\Gamma(2-\alpha)}.
\end{align}
For the truncation error of the fast algorithm, similar to the proof in \cite{Cao20,Jiang17}, it is straightforward to conclude the following lemma.
\begin{lem}\label{lem2.3}
Suppose $|\frac{du}{dt}|\leq C(1+t^{\alpha-l})$ for $l=0,1,2$ and $0<\alpha<1$ and let $\varepsilon $ be the desired precision, then we have the truncation error of the fast schemes (\ref{eq2.22})-(\ref{eq2.23})
\begin{align*} 
|R_3^n|=\left|{_0^CD_t^{(\alpha,\rho)}}u(t_n)- {^{FC}_0\mathbb{D}_t^{(\alpha,\rho)}}u(t_n)\right|\leq C \left( n^{-\min\{2-\alpha,r\alpha\}}+\varepsilon\right).
\end{align*}
\end{lem}
In the subsequent sections, we adapt the tempered operator into different models to observe the behaviour of the tempered models.


\section{Application I: Tempered Bloch equations}\label{Sec3}

Ordinary matter all has nuclei with induced magnetic moments, leading to a nuclear paramagnetic polarisation in a constant magnetic field when an equilibrium is established. Imposing an external radiofrequency field to the constant field at a right angle can cause the Larmor precession of the moments around the constant field, which can induce an electromotive force in a coil \cite{Bloch1,Bloch2}. Furthermore, the induced electromotive force can be transferred into visible signals. This is the principle of nuclear magnetic resonance (NMR) and magnetic resonance imaging (MRI). NMR has been widely used to analyse complex biological materials in chemistry, medicine, and engineering. The phenomenological Bloch equation is utilised to capture the magnetisation dynamics. The empirical vector form of the Bloch equation is
\begin{align*} 
\frac{d{\bf M}(t)}{dt}=\gamma {\bf M}(t)\times {\bf B}(t) +\frac{M_0-M_z(t)}{T_1}{\bf k}-\frac{M_x(t){\bf i}+M_y(t){\bf j}}{T_2},
\end{align*}
where $\mathbf{M}=[M_{x},M_{y},M_{z}]^{T}$ is the magnetization, $\mathbf{B}$ is the static magnet field, $\mathbf{B}(t)=B_{0}\mathbf{k}$, $M_{0}$ is the equilibrium magnetization, $\gamma $ is the gyromagnetic ratio with a relation to the Larmor frequency $\omega _{0}=\gamma B_{0}$. The relaxation terms refer to the time return to equilibrium, where $T_{1}$ is the `longitudinal' relaxation time and $T_{2}$ is the `transversal' relaxation time. To investigate heterogeneous, porous and complex materials exhibiting memory, Magin et al. \cite{Magin09, Magin20} generalise the Bloch equation to the time-fractional Bloch equation. From the view of CTRW, the time-fractional operator relates to an infinite waiting time. Since the waiting time for a water proton cannot be arbitrarily long, we modify the equation as the tempered Bloch equation with finite waiting time, which is
written in the form
\begin{align*}
& \mu_1^{\alpha-1}{^{C}_0D_t^{(\alpha,\rho)}}M_z(t)=\frac{M_0-M_z(t)}{T_1},\\
& \mu_2^{\alpha-1}{^{C}_0D_t^{(\alpha,\rho)}}M_x(t)=\omega_0M_y(t)-\frac{M_x(t)}{T_2},\\
& \mu_2^{\alpha-1}{^{C}_0D_t^{(\alpha,\rho)}}M_y(t)=-\omega_0M_x(t)-\frac{M_y(t)}{T_2}.
\end{align*}
These equations can be recast into
\begin{align} \label{eq3.2}
{^{C}_0D_t^{(\alpha,\rho)}}M_z(t)&=\frac{M_0-M_z(t)}{T'_1},\\\label{eq3.3}
{^{C}_0D_t^{(\alpha,\rho)}}M_x(t)&=\varpi_0M_y(t)-\frac{M_x(t)}{T'_2},\\\label{eq3.4}
{^{C}_0D_t^{(\alpha,\rho)}}M_y(t)&=-\varpi_0M_x(t)-\frac{M_y(t)}{T'_2},
\end{align}
where $\mu _{1}$ and $\mu _{2}$ are constants to preserve units, $\alpha $ is the fractional index, $\rho $ is the tempered parameter, $\varpi _{0}=\omega _{0}\mu _{2}^{1-\alpha }$, $T_{1}^{\prime }=\mu _{1}^{\alpha -1}T_{1}$ and $T_{2}^{\prime }=\mu _{2}^{\alpha -1}T_{2}$. When $\rho =0$, the system recovers the time-fractional Bloch equation proposed in \cite{Magin09}. We consider the analytical and numerical solutions for the system.
\subsection{Analytical solution}

Applying the Laplace transform to (\ref{eq3.2}) leads to
\begin{align*}
\overline{M_z}(s)=\frac{(s+\rho)^{\alpha-1}}{(s+\rho)^{\alpha}+k_1}M_z(0)
+\frac{k_2}{s[(s+\rho)^{\alpha}+k_1]},
\end{align*}
where $k_1=\frac{1}{T'_1}$ and $k_2=\frac{M_0}{T'_1}$. Using the property $\mathcal{L}\{e^{\rho t}f(t)\}=\bar{f}(s-\rho)$, $\mathcal{L}\{t^{\alpha-1}E_{\alpha,\alpha}(-kt^\alpha)\}=\frac{1}{s^\alpha+k}$ and $\mathcal{L}^{-1}\{\bar{f}(s)\bar{g}(s)\}=\int_0^tf(\eta)g(t-\eta)d\eta$ and applying the inverse Laplace transform, we can derive
\begin{align}
M_z(t)=&M_z(0)e^{-\rho t}E_{\alpha}(-k_1t^{\alpha})+k_2\int_0^te^{-\rho \eta}\eta^{\alpha-1}E_{\alpha,\alpha}(-k_1\eta^{\alpha})d\eta\nonumber\\\label{eq3.5}
=&M_z(0)e^{-\rho t}E_{\alpha}\left(-\frac{t^{\alpha}}{T'_1}\right)+\frac{M_0}{T'_1}\int_0^te^{-\rho \eta}\eta^{\alpha-1}E_{\alpha,\alpha}\left(-\frac{\eta^{\alpha}}{T'_1}\right)d\eta.
\end{align}
Suppose that $M_{+}(t)=M_x(t)+iM_y(t)$ with $M_{+}(0)=M_x(0)+iM_y(0)$. Combining (\ref{eq3.3}) and (\ref{eq3.4}) gives
\begin{align} \label{eq3.6}
{^{C}_0D_t^{(\alpha,\rho)}}M_+(t)&=-i\varpi_0M_+(t)-\frac{M_+(t)}{T'_2}.
\end{align}
Applying the Laplace transform and the inverse Laplace transform to  (\ref{eq3.6}), we can obtain
\begin{align} \label{eq3.7}
M_+(t)=M_{+}(0)e^{-\rho t}E_{\alpha}(-k_3t^{\alpha}),
\end{align}
where $k_3=i\varpi_0+\frac{1}{T'_2}$. As the analytical solutions (\ref{eq3.5}) and (\ref{eq3.7}) involve the Mittag-Leffler function and its integral, which is very challenging to calculate, we need to resort to a numerical solution of the problem.
\subsection{Numerical solution}
We use the L1 scheme on graded mesh to solve the tempered problem (\ref{eq3.2})-(\ref{eq3.4}). Define ${\bf M}^n$ as the numerical approximation of ${\bf M}(t_n)$. At $t=t_n$, we can obtain the following numerical scheme:
\begin{align*} 
{_0^C\mathbb{D}_t^{(\alpha,\rho)}}M_z^n&=\frac{M_0-M_z^n}{T'_1},\\
{_0^C\mathbb{D}_t^{(\alpha,\rho)}}M_x^n&=\varpi_0M_y^n-\frac{M_x^n}{T'_2},\\
{_0^C\mathbb{D}_t^{(\alpha,\rho)}}M_y^n&=-\varpi_0M_x^n-\frac{M_y^n}{T'_2}.
\end{align*}

\begin{figure}[H]
\centering
\subfloat[Varying $\alpha$  ]{
\label{fig1a}
\begin{minipage}[t]{0.46\textwidth}
\centering
\scalebox{0.45}[0.45]{\includegraphics{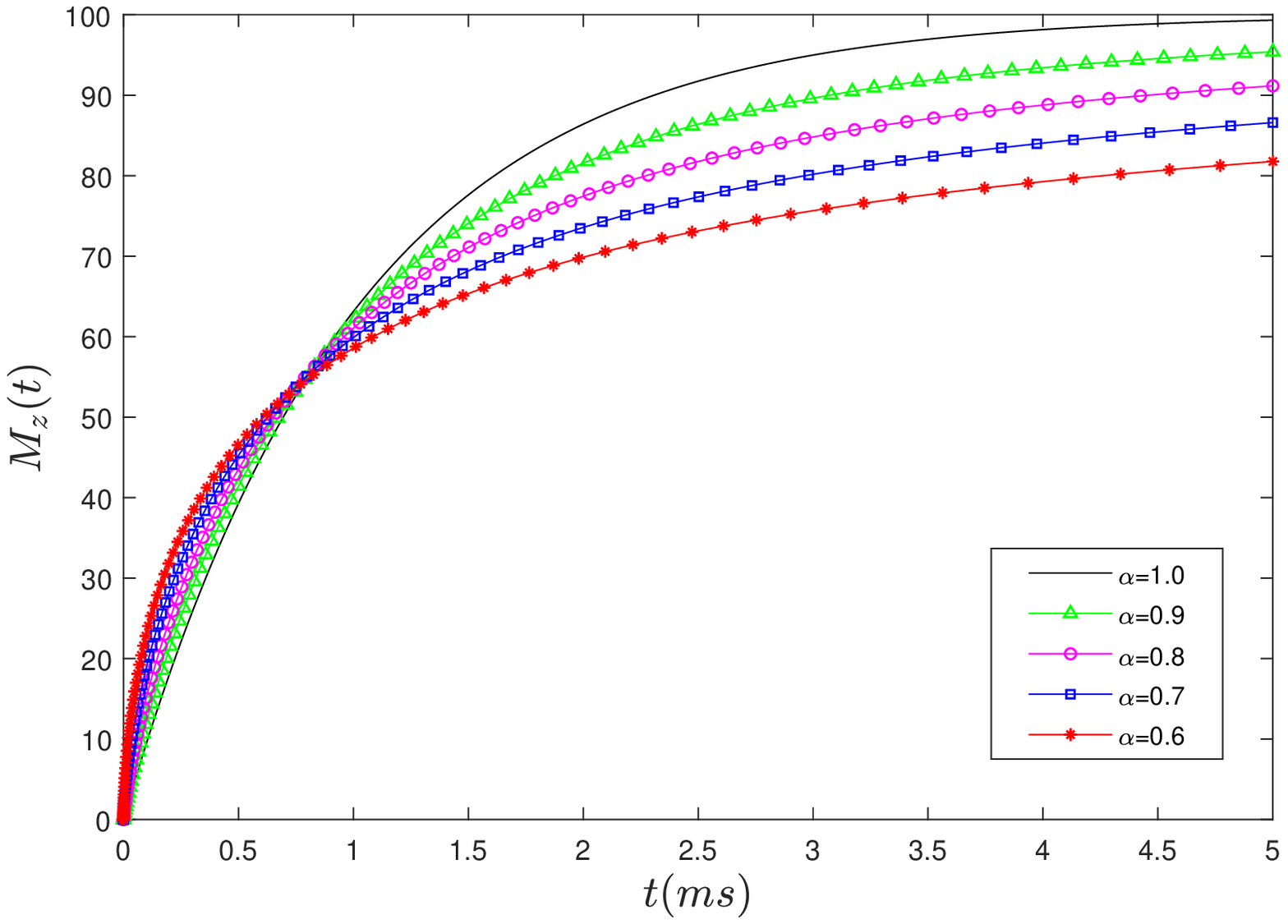}}
\end{minipage}
}
\subfloat[Varying $\rho$ ]{
\label{fig1b}
\begin{minipage}[t]{0.46\textwidth}
\centering
\scalebox{0.49}[0.49]{\includegraphics{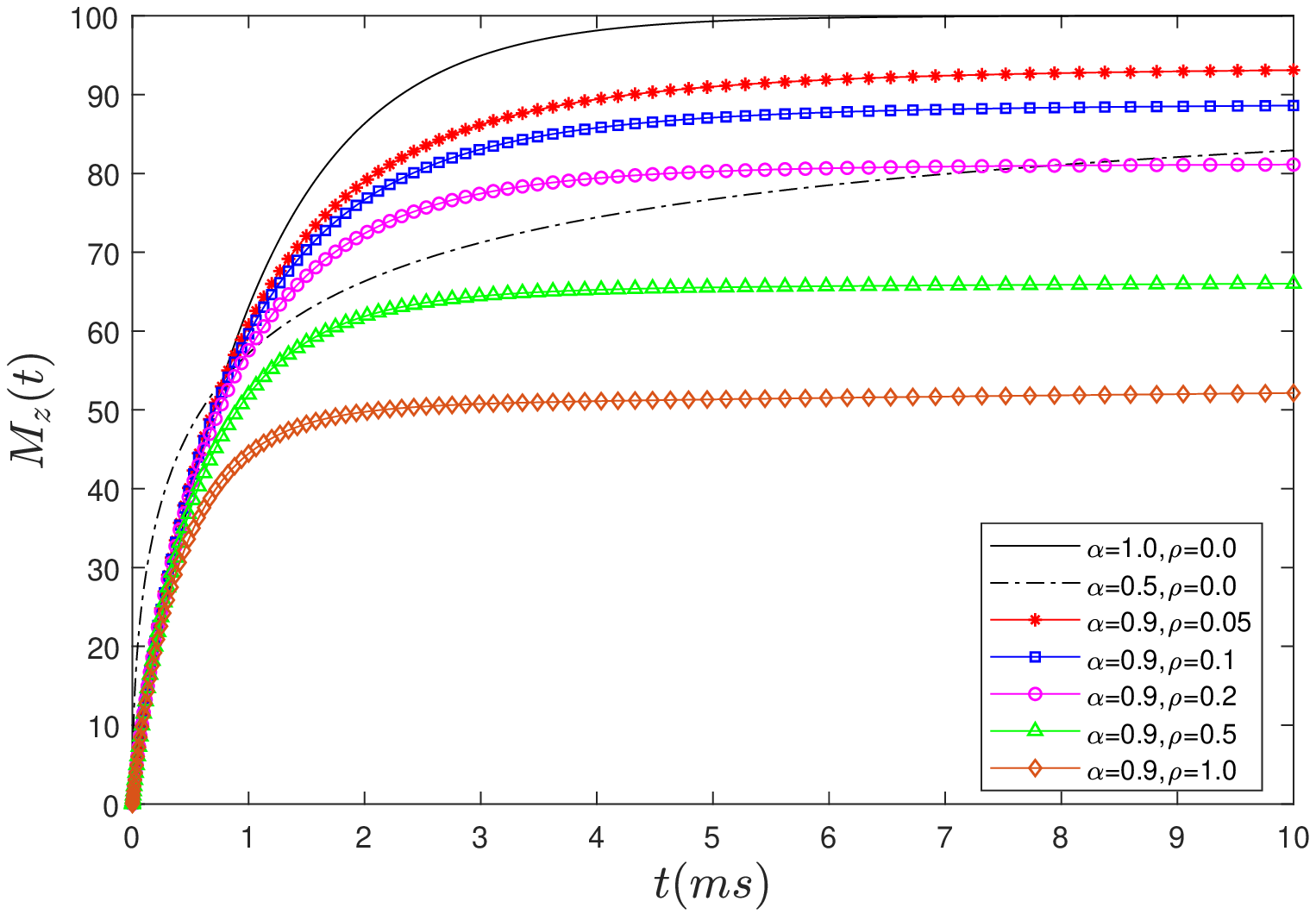}}
\end{minipage}
}
\caption{The evolution of $M_z(t)$ with different $\alpha$ (Figure a) and different $\rho$ (Figure b), where the parameters used are $M_z(0)=0$, $M_0=100$, $T_1'=1(ms)^{\alpha}$.}
\label{fig1}
\end{figure}
\begin{figure}[H]
\begin{center}
\scalebox{0.5}[0.5]{\includegraphics{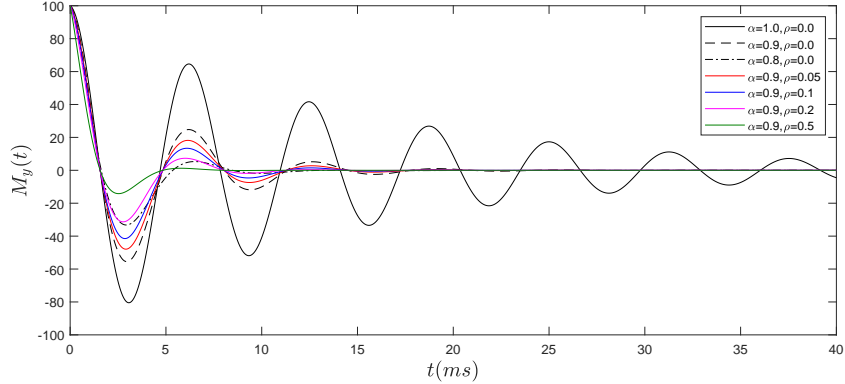}}
\caption{The evolution of $M_y(t)$ with different $\alpha$ and $\rho$, where the parameters used are $T_2'=20(ms)^{\alpha}$ $\tilde{f}_0=\frac{\varpi_0}{2\pi}=160Hz$, $M_x(0)=0$, $M_y(0)=100$.}
\label{fig2}
\end{center}
\end{figure}

\subsection{Numerical examples}

Some numerical results are shown in Figures (\ref{fig1})-(\ref{fig3}). Figure \ref{fig1a} shows that the fractional order $\alpha$ enhances the relaxation at a small time scale and then delays the relaxation later compared to the classical case $\alpha =1$. The smaller the fractional index is, the slower it converges to its final asymptotic value. Figure \ref{fig1b} illustrates that the tempered parameter $\rho $ can accelerate the process to reach its asymptotic value. Contrary to the fractional order $\alpha $, the larger $\rho $ is, the more rapid it converges to its maximum value. In addition, with a large $\alpha $ and moderate $\rho $ ($\alpha =0.9$, $\rho =0.2$), we can obtain a similar final state by choosing small $\alpha $ only ($\alpha =0.5$, $\rho =0.0$). This facilitates avoiding choosing too small $\alpha $ since a small fractional index means strong
heterogeneity or low regularity of the solution.

The effects of $\alpha $ and $\rho $ on $M_{y}(t)$ is presented in Figure \ref{fig2}. We can see that, for the classical case $\alpha =1$, $\rho =0$, it needs a long relaxation time to decay to zero. The fractional order boosts the decay. The smaller $\alpha $ is, the faster it decays. The tempered parameter $\rho $ promotes the decay further. We can observe that it decays faster with $\alpha =0.9$, $\rho =0.2$ than with $\alpha =0.8$, $\rho =0.0$, which is similar to the discussion above. Figure \ref{fig3} displays the effects of $\rho $ on the evolution of $M_{x}(t)$ versus $M_{y}(t)$. Similarly, the tempered parameter advances the decay process to zero.

\begin{figure}[H]
\begin{center}
\scalebox{0.5}[0.5]{\includegraphics{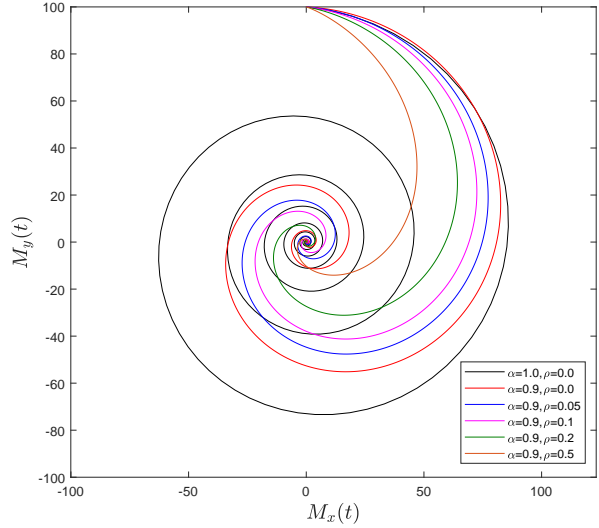}}
\caption{The evolution of $M_x(t)$ versus $M_y(t)$ with different $\rho$, where the parameters used are $T_2'=20(ms)^{\alpha}$ $\tilde{f}_0=\frac{\varpi_0}{2\pi}=160Hz$, $M_x(0)=0$, $M_y(0)=100$.}
\label{fig3}
\end{center}
\end{figure}

\begin{figure}[H]
\centering
\subfloat[ ]{
\label{figN3a}
\begin{minipage}[t]{0.46\textwidth}
\centering
\scalebox{0.42}[0.42]{\includegraphics{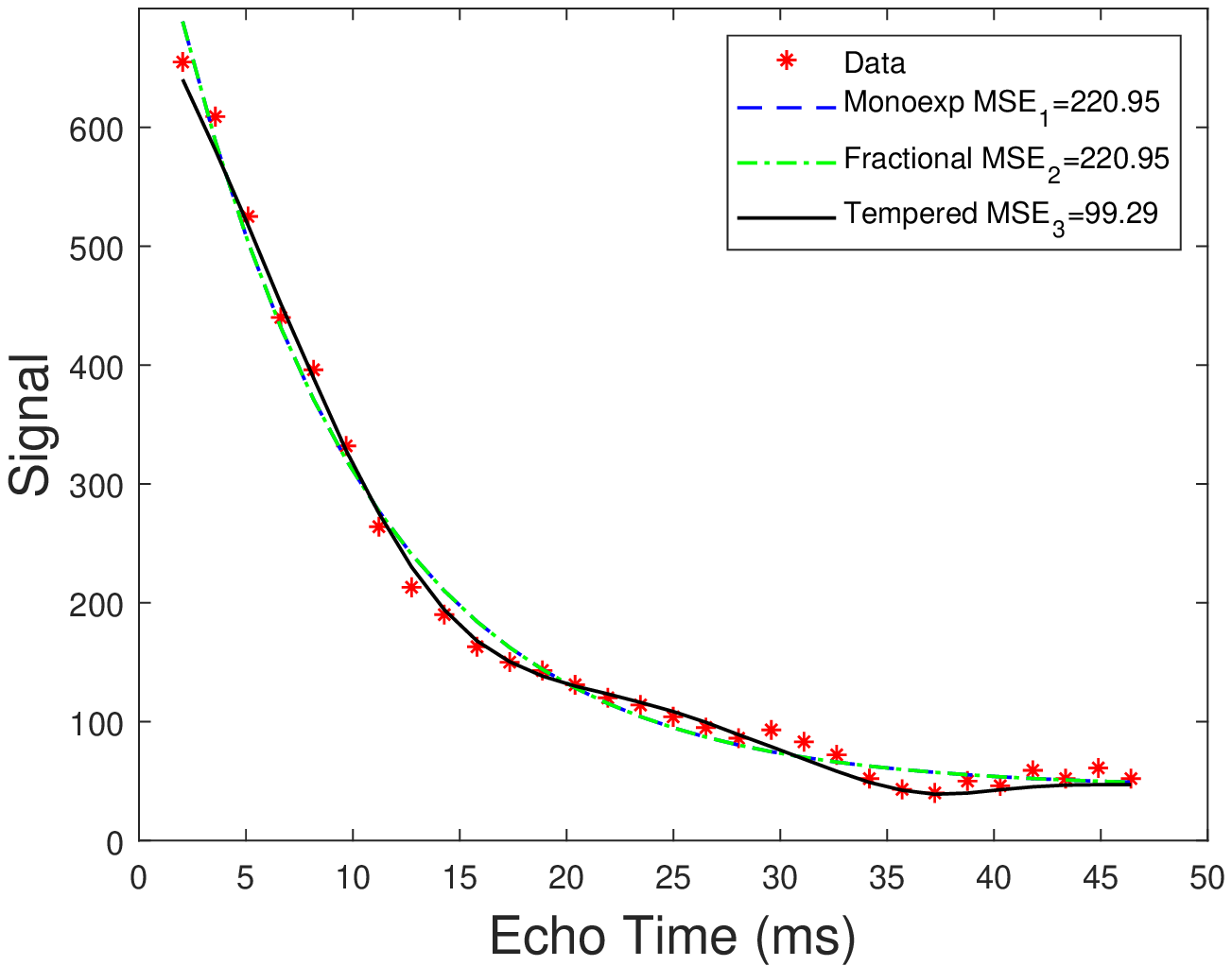}}
\end{minipage}
}
\subfloat[]{
\label{figN3b}
\begin{minipage}[t]{0.46\textwidth}
\centering
\scalebox{0.42}[0.42]{\includegraphics{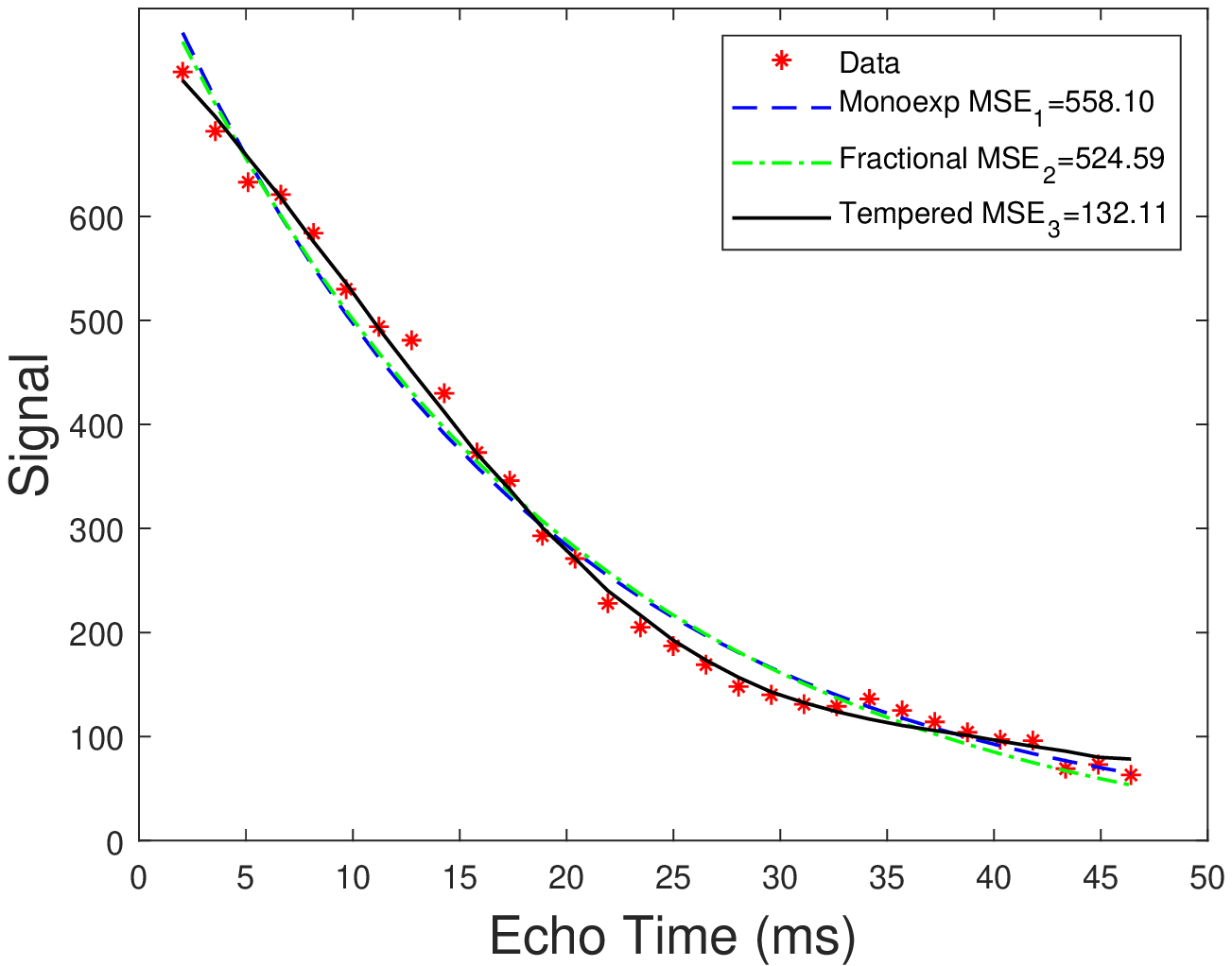}}
\end{minipage}
}
\caption{Voxel-level data fittings of the classical model (Monoexp), time-fractional model (Fractional) and tempered time-fractional model (Tempered). Here we consider two different voxels (locations (a) (132, 55, 47) and (b) (80, 60, 68)) in different brain regions, of which the data is extracted from \cite{Qin17}. The mean square error (MSE) is defined as $\sum_{j=1}^n(y'_j-y_j)^2/n$, where $y_j$ is the real data, $y'_j$ is the predicted data and $n$ is the number of data.}
\label{figN3}
\end{figure}

The Bloch equation is related to the MRI signal by $S(t)=A_0\sqrt{M_x^2(t)+M_y^2(t)}+C$, where $A_0$ is the amplitude of the signal and $C$ is a constant accounting for the background noise \cite{Qin16}. Figure \ref{figN3} shows the voxel-level temporal fitting based on three models, namely the classical mono-exponential model ($A_0\exp(-t/T_2)+C$), the time-fractional model ($A_0E_{\alpha}(-t^{\alpha}/T_2)+C$) \cite{Magin09} and our tempered time-fractional model. The experimental data is extracted from the paper \cite{Qin17}. The `lsqcurvefit' MATLAB routine is used to perform all the data fittings with a maximum number of iteration and a relative tolerance set as $1\times 10^7$ and $1\times 10^{-7}$, respectively. The specific fitting procedure can be referred to \cite{Qin16}. This procedure is generally followed here to determine the initial conditions for our model with $A_0$, $T_2$ and $C$ based on the results from the exponential model and $\alpha$ from the fractional model. These parameters are constrained in ranges to $70-130\%$ of their starting values. $\omega$ and $\rho$ are new parameter to the other two models, which are initialized to $10$ and $0.1$ with ranges confined in $0-250 $ and $0-1$. The mean square error (MSE) is adopted to measure the quality of the fitting. From Figure \ref{figN3} we can see that for the signal at position (a), both the classical model and the time-fractional model have a large MSE while the tempered model has a clear improvement in the accuracy of fitting the voxel-level MRI data with less than a half MSE. For the data at position (b), the classical model does not fit well with an MSE 558.10. Although the time-fractional model improves the fitting little, the MSE is still large with a value of (524.59). In comparison, the MSE for the tempered time-fractional model reduces significantly to 132.11. We conclude that the tempered time-fractional model is effective in fitting the MRI signal.


\section{Application II: Tempered diffusion problems }\label{Sec4}

In the CTRW model, when a Poissonian waiting time pdf together with a Gaussian jump length pdf is applied, the standard diffusion equation can be derived, describing the Brownian motion. When the finite waiting time pdf is replaced with a divergent long-tailed waiting time pdf, it leads to the time-fractional diffusion equation used to characterise the anomalous diffusion \cite{Metzler}. Different from the standard diffusion with mean squared displacement $\langle x^{2}(t)\rangle \sim t$, the mean squared displacement of time-fractional diffusion has the property $\langle x^{2}(t)\rangle \sim t^{\alpha }$, where $0<\alpha <1$ is the fractional order. In this section, we consider a tempered long-tailed waiting time pdf to solve a tempered diffusion problem
\begin{align} \label{eq4.1}
{^{C}_0D_t^{(\alpha,\rho)}}u(x,t)=D\frac{\partial^2u(x,t)}{\partial x^2}+f(x,t),\quad x\in(0,l),~ t\in(0,T],
\end{align}
subject to the initial and boundary conditions
\begin{align} \label{eq4.2}
u(x,0)=\psi(x),\quad x\in[0,l],\quad u(0,t)=u(l,t)=0,\quad t\in(0,T].
\end{align}

\subsection{The mean squared displacement}\label{sec4.1}

To study the mean squared displacement of the problem, we suppose $f(x,t)=0$, $\psi(x)=\delta(x)$. Applying the Laplace transform and Fourier transform successively to obtain
\begin{align*} 
\widehat{\overline{u}}(\omega,s)=\frac{(s+\rho)^{\alpha-1}}{(s+\rho)^{\alpha}+D\omega^2}.
\end{align*}
As $-\frac{\partial^2 \widehat{\overline{u}}(\omega,s)}{\partial \omega^2}\big|_{\omega=0}=-\frac{2D}{(s+\rho)^{\alpha+1}}$, we derive
\begin{align*} 
\langle x^2(t) \rangle =e^{-\rho t}\frac{2Dt^\alpha}{\Gamma(1+\alpha)}.
\end{align*}
Different to the mean squared displacement of the time-fractional diffusion in \cite{Metzler}, one tempered factor $e^{-\rho t}$ is added.

\subsection{Analytical solution}\label{sec4.2}

We first use the finite Fourier transform technique to solve the problem  (\ref{eq4.1})-(\ref{eq4.2}). The associated eigenvalue problem that needs to be solved is
\begin{align*} 
-\frac{d^2 \varphi}{dx^2}=\lambda^2\varphi,\quad \varphi(0)=0,\quad \varphi(l)=0,
\end{align*}
to which the solution is $\lambda_n^2=\frac{n^2\pi^2}{l^2}$ for $n=1,2,\ldots$ and $\varphi_n(x)=\sqrt{\frac{2}{l}}\sin\left( \lambda_nx\right)$. Define the finite Fourier transform
$\widetilde{u}(\lambda_n,t):=\langle u, \varphi_n\rangle=\int_0^lu(x,t)\varphi_n(x)dx$. Applying it to (\ref{eq4.1}), we obtain
\begin{align} \label{eq4.6}
\left\langle {^{C}_0D_t^{(\alpha,\rho)}}u(x,t), \varphi_n\right\rangle=D\left\langle \frac{\partial^2u(x,t)}{\partial x^2} ,\varphi_n\right\rangle+\langle f(x,t),\varphi_n\rangle.
\end{align}
As $D\left\langle \frac{\partial^2u(x,t)}{\partial x^2} ,\varphi_n\right\rangle=-D\left\langle \frac{d^2 \varphi_n}{d x^2} ,u\right\rangle=-D\lambda_n^2\langle u, \varphi_n\rangle$, (\ref{eq4.6}) can be recast into
\begin{align} \label{eq4.7}
{^{C}_0D_t^{(\alpha,\rho)}}\widetilde{u}=-D\lambda_n^2\widetilde{u}+\widetilde{f}.
\end{align}
Applying the Laplace transform to (\ref{eq4.7}) and denoting $\overline{\widetilde{u}}(\lambda_n,s)=\mathcal{L}\{\widetilde{u}(\lambda_n,t)\}$, we can derive
\begin{align} \label{eq4.8}
\overline{\widetilde{u}}=\frac{(s+\rho)^{\alpha-1}\widetilde{u}(0)}{(s+\rho)^{\alpha}+D\lambda_n^2}
+\frac{\overline{\widetilde{f}}}{(s+\rho)^{\alpha}+D\lambda_n^2}.
\end{align}
Imposing the inverse Laplace transform to (\ref{eq4.8}) leads to
\begin{align*} 
\widetilde{u}(\rho_n,t)=\widetilde{u}(0)e^{-\rho t}E_{\alpha}(-D\lambda_n^2 t^\alpha)
+\int_0^te^{-\rho \eta}\eta^{\alpha-1}E_{\alpha,\alpha}\left(-D\lambda_n^2 \eta^\alpha\right)
\widetilde{f}(t-\eta)d\eta.
\end{align*}
Then the analytical solution can be derived as
\begin{align} \label{eq4.10}
u(x,t)=\sqrt{\frac{2}{l}}\sum_{n=1}^{\infty}\left[ \widetilde{u}(0)e^{-\rho t}E_{\alpha}(-D\lambda_n^2 t^\alpha)+\int_0^te^{-\rho \eta}\eta^{\alpha-1}E_{\alpha,\alpha}\left(-D\lambda_n^2 \eta^\alpha\right)
\widetilde{f}(t-\eta)d\eta\right]\sin\left( \frac{n\pi x}{l}\right),
\end{align}
where $\widetilde{u}(0)=\langle \psi, \varphi_n\rangle$, $\widetilde{f}=\langle f, \varphi_n\rangle$.

\subsection{Numerical solution}\label{sec4.3}

\subsubsection{Numerical scheme}
Now we consider the numerical solution to (\ref{eq4.1})-(\ref{eq4.2}). Firstly, we denote a uniform mesh in space domain. Define $h=\frac{l}{M}$, $x_i=ih$, $i=0,1,2,\ldots,M$, where $M$ is a positive integer. For the space Laplacian operator, we use the standard second-order central difference scheme to approximate:
\begin{align*} 
\frac{\partial^2u(x_i,t_n)}{\partial x^2}=\frac{u(x_{i-1},t_n)-2u(x_i,t_n)+u(x_{i+1},t_n)}{h^2}+\mathcal{O}(h^2):=\delta_x^2u(x_i,t_n)+\mathcal{O}(h^2).
\end{align*}
Let $u_i^n$ be the numerical solution to $u(x_i,t_n)$. Then the numerical scheme to (\ref{eq4.1})-(\ref{eq4.2}) is
\begin{align} \label{eq4.12}
&{_0^C\mathbb{D}_t^{(\alpha,\rho)}}u_i^n=D\delta_x^2u_i^n+f_i^n,\quad 1\leq i\leq M-1,~1\leq n\leq N,\\\label{eq4.13}
&u_0^n=u_M^n=0,\quad {\rm{for}}~0<n\leq N,\\\label{eq4.14}
& u_i^0=\psi(x_i), \quad {\rm{for}}~0\leq i\leq M.
\end{align}
And the fast numerical scheme to (\ref{eq4.1})-(\ref{eq4.2}) is
\begin{align} \label{eq4.15}
&{^{FC}_0\mathbb{D}_t^{(\alpha,\rho)}}u_i^n=D\delta_x^2u_i^n+f_i^n,\quad 1\leq i\leq M-1,~1\leq n\leq N,\\\label{eq4.16}
&u_0^n=u_M^n=0,\quad {\rm{for}}~0<n\leq N,\\\label{eq4.17}
& u_i^0=\psi(x_i), \quad {\rm{for}}~0\leq i\leq M.
\end{align}

\subsubsection{Stability}
\begin{thm}\label{stability}
The numerical scheme (\ref{eq4.12}) is unconditional stable and it holds that
\begin{align*}
||u^n||_{\infty}\leq ||u^0||_{\infty}+\tau_n^\alpha\Gamma(2-\alpha)\sum_{j=1}^n\theta_{n,j}||f^j||_{\infty},
\end{align*}
where $\theta_{n,n}=1$,  $\theta_{n,j}=\sum_{k=1}^{n-j}\tau_{n-k}^{\alpha}(d_{n,k}-d_{n,k+1})\theta_{n-k,j}$.
\end{thm}
\begin{proof}
According to Lemma 4.2 in \cite{Stynes17}, it is straightforward to derive
\begin{align*}
||e^{\rho t_n} u^n||_{\infty} \leq ||u^0||_{\infty}+\tau_n^\alpha\Gamma(2-\alpha)\sum_{j=1}^n\theta_{n,j}||e^{\rho t_j}f^j||_{\infty},
\end{align*}
which leads to
\begin{align*}
||u^n||_{\infty}&\leq  e^{-\rho t_n}||u^0||_{\infty}+e^{-\rho t_n}\tau_n^\alpha\Gamma(2-\alpha)\sum_{j=1}^n\theta_{n,j}||e^{\rho t_j}f^j||_{\infty}\\
&\leq ||u^0||_{\infty}+\tau_n^\alpha\Gamma(2-\alpha)\sum_{j=1}^n\theta_{n,j}||f^j||_{\infty}.
\end{align*}
The proof is completed.
\end{proof}
\subsubsection{Convergence}
\begin{thm}
Define $u^n$ and $U^n$ as the exact and numerical solution vector, respectively. Then there exists a positive constant $C$ independent of $\tau$ and $h$ such that
\begin{align*}
||u^n-U^n||_{\infty}\leq C(h^2+T^\alpha N^{-\min\{2-\alpha,r\alpha\}}).
\end{align*}
\end{thm}
\begin{proof}
The truncation error of (\ref{eq4.12}) at $(x_i,t_n)$ is
\begin{align*}
|\mathfrak{R}^n|\leq C(h^2+T^\alpha n^{-\min\{2-\alpha,r\alpha\}}).
\end{align*}
Invoking Theorem \ref{stability}, we have
\begin{align*}
\max|u(x_i,t_n)-u^n_i|&\leq C \tau_n^\alpha\Gamma(2-\alpha)\sum_{j=1}^n\theta_{n,j}||\mathfrak{R}^j||_{\infty}\\
&\leq C\tau_n^\alpha\Gamma(2-\alpha)\sum_{j=1}^n\theta_{n,j} (h^2+T^\alpha j^{-\min\{2-\alpha,r\alpha\}}\\
&\leq CT^\alpha(h^2+ N^{-\min\{2-\alpha,r\alpha\}}),
\end{align*}
where the following inequality \cite{Stynes17} has been used
\begin{align*}
\tau_n^{\alpha}\sum_{j=1}^nj^{-\beta}\theta_{n,j}\leq \frac{T^\alpha N^{-\beta}}{1-\alpha}.
\end{align*}
The proof is completed.
\end{proof}
Similarly, we can obtain the convergence of the fast numerical scheme (\ref{eq4.15})-(\ref{eq4.17}).
\begin{thm}
Define $u^n$ and $U^n$ as the exact and numerical solution vector, respectively. Then there exists a positive constant $C$ independent of $\tau$ and $h$ such that
\begin{align*}
||u^n-U^n||_{\infty}\leq C(h^2+T^\alpha N^{-\min\{2-\alpha,r\alpha\}}+\epsilon).
\end{align*}
\end{thm}

\subsubsection{Numerical examples}

In this section, for problem (\ref{eq4.1})-(\ref{eq4.2}), we consider the case $f(x,t)=0$, $u(x,0)=\sin (x)$ in the domain $[0,\pi ]$. Using (\ref{eq4.10}), it is straightforward to derive the analytical solution $u(x,t)=e^{-\rho t}E_{\alpha }(-Dt^{\alpha })\sin (x)$. We apply the numerical scheme (\ref{eq4.12})-(\ref{eq4.14}) and the fast numerical scheme (\ref{eq4.15})-(\ref{eq4.17}) to solve the problem. All of the computations were conducted using MATLAB R2018a on a DELL desktop with the configuration:
Intel(R) Core(TM) i7-6700 CPU3.40GHz and RAM 16.0 GB. Firstly, a comparison between the numerical solution and the exact solution at different time is plotted in Figure \ref{fig4}, where the parameters chosen are $M=40$, $N=M^{2}$, $D=0.5$, $\alpha =0.8$, $\rho =0.5$. It can be observed that the numerical solution agrees with the exact solution very well.

\begin{table}[H]
\caption{The error and convergence order of the L1 scheme (\ref{eq4.12}) and the fast numerical scheme (\ref{eq4.15}) for different $\alpha$ and $N$ at $t=1$, where the parameters used are $M=2^{11}$, $D=2$, $\rho=0.5$, $r=\frac{2(2-\alpha)}{\alpha}$, $\epsilon=10^{-9}$.}
\label{tab5}
\centering
\begin{tabular}{c|  c c| c c}
\toprule
\multirow{2}{*}{\centering $N$ (L1 scheme)}
 &\multicolumn{2}{c|}{$\alpha=0.4$} &\multicolumn{2}{c}{$\alpha=0.8$} \\
\cmidrule{2-5}
& Error & Order&  Error   &Order    \\
\midrule
$80$   & 2.0069E-04     &        &  1.0678E-03      &          \\
$160$  & 6.7734E-05     &  1.57  &  4.6677E-04      &  1.19    \\
$320$  & 2.2806E-05     &  1.57  &  2.0363E-04      &  1.20    \\
$640$  & 7.6305E-06     &  1.58  &  8.8752E-05      &  1.20    \\
$1280$ & 2.5381E-06     &  1.59  &  3.8676E-05      &  1.20    \\
$2560$ & 8.3418E-07     &  1.61  &  1.6861E-05      &  1.20    \\
\midrule
 \multirow{2}{*}{\centering $N$ (Fast scheme)}
  &\multicolumn{2}{c|}{$\alpha=0.4$} &\multicolumn{2}{c}{$\alpha=0.8$} \\
\cmidrule{2-5}
& Error  &  Order  &  Error   &  Order     \\
\midrule
$80$   & 2.3540E-04     &        &  1.1033E-03      &          \\
$160$  & 8.2692E-05     &  1.51  &  4.8357E-04      &  1.19    \\
$320$  & 2.8682E-05     &  1.53  &  2.1130E-04      &  1.19    \\
$640$  & 9.8370E-06     &  1.54  &  9.2195E-05      &  1.20    \\
$1280$ & 3.3412E-06     &  1.56  &  4.0201E-05      &  1.20    \\
$2560$ & 1.1159E-06     &  1.58  &  1.7532E-05      &  1.20    \\
\bottomrule
\end{tabular}
\end{table}

\begin{table}[H]
\caption{The error and CPU time comparison between the L1 scheme (\ref{eq4.12}) and the fast numerical scheme (\ref{eq4.15}) for different $\alpha$ and $M$ with $N=M^2$ at $t=1$, where the parameters utilised are $D=2$, $\rho=0.5$, $r=\frac{2(2-\alpha)}{\alpha}$, $\epsilon=10^{-9}$.}
\label{tab6}
\centering
\begin{tabular}{c| c c c| cc c}
\toprule
\multirow{3}{*}{\centering $M$ (L1 scheme)}
 &\multicolumn{3}{c|}{$\alpha=0.4$} &\multicolumn{3}{c}{$\alpha=0.8$} \\
\cmidrule{2-7}
         &  Error         & Order & CPU time(s) & Error            & Order & CPU time(s)   \\
\midrule
$20$    & 3.4560E-04     &        & 0.09          &  5.5454E-04      &       &  0.10  \\
$40$    & 8.4137E-05     &  2.04  & 1.03          &  1.2924E-04      &  2.10 &  1.04  \\
$80$    & 2.0748E-05     &  2.02  & 15.92         &  3.0522E-05      &  2.08 &  15.28  \\
$160$   & 5.1245E-06     &  2.02  & 257.62        &  7.2913E-06      &  2.07 &  256.84  \\
$320$   & 1.2457E-06     &  2.04  & 4836.06       &  1.7585E-06      &  2.05 &  4949.47 \\
\midrule
\multirow{3}{*}{\centering $M$ (Fast scheme) }
 &\multicolumn{3}{c|}{$\alpha=0.4$} &\multicolumn{3}{c}{$\alpha=0.8$} \\
\cmidrule{2-7}
         &  Error         & Order & CPU time(s) & Error            & Order & CPU time(s)   \\
\midrule
$20$    & 3.4995E-04     &       & 0.30         &  5.6048E-04      &       &  0.18  \\
$40$    & 8.4708E-05     &  2.05 & 1.45         &  1.3041E-04      &  2.10 &  0.69  \\
$80$    & 2.0799E-05     &  2.03 & 6.50         &  3.0747E-05      &  2.08 &  3.04  \\
$160$   & 5.0624E-06     &  2.04 & 34.64        &  7.3341E-06      &  2.07 &  16.11  \\
$320$   & 1.0722E-06     &  2.24 & 209.33       &  1.7667E-06      &  2.05 &  95.86 \\
\bottomrule
\end{tabular}
\end{table}

\begin{figure}[H]
\begin{center}
\scalebox{0.5}[0.5]{\includegraphics{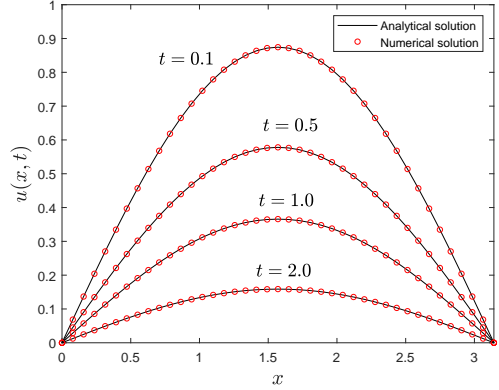}}
\caption{The comparison between the numerical solution and the exact solution at different time, where the parameters chosen are $M=40$, $N=M^2$, $D=0.5$, $\alpha=0.8$, $\rho=0.5$.}
\label{fig4}
\end{center}
\end{figure}

\begin{figure}[H]
\centering
\subfloat[Varying $\alpha$  ]{
\label{fig5a}
\begin{minipage}[t]{0.46\textwidth}
\centering
\scalebox{0.45}[0.45]{\includegraphics{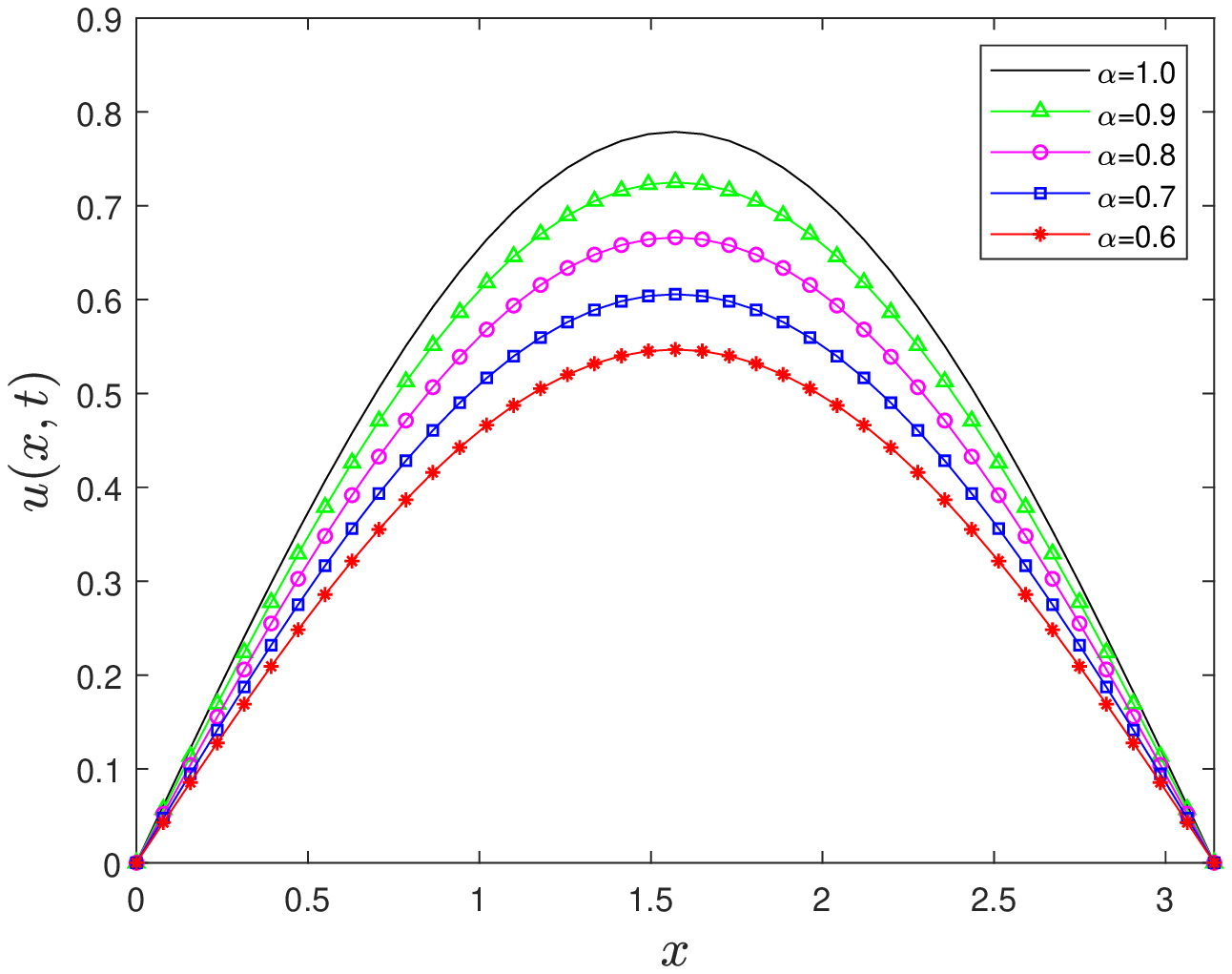}}
\end{minipage}
}
\subfloat[Varying $\rho$ ]{
\label{fig5b}
\begin{minipage}[t]{0.46\textwidth}
\centering
\scalebox{0.36}[0.36]{\includegraphics{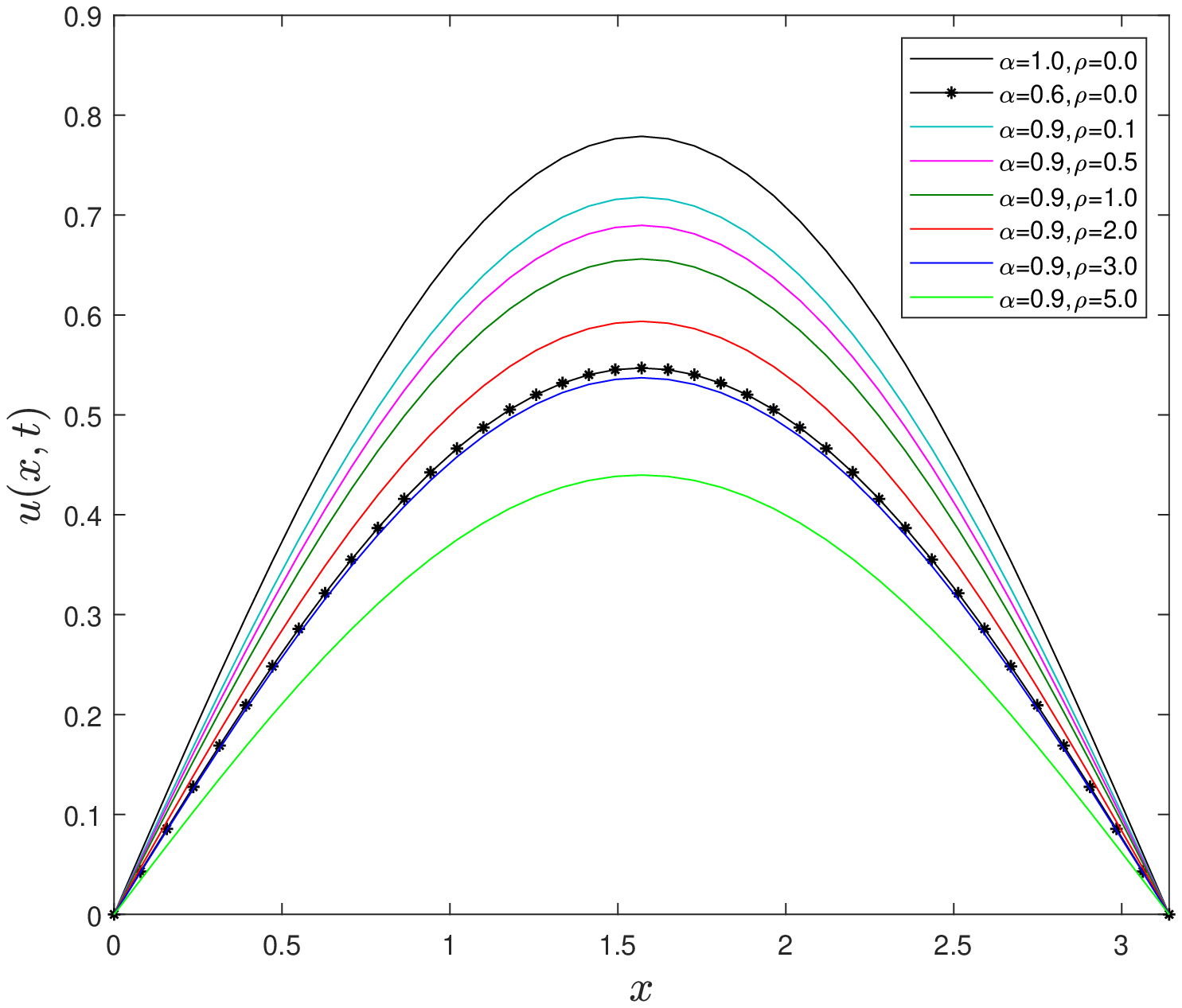}}
\end{minipage}
}
\caption{The impacts of the fractional index $\alpha$ (Figure a) and tempered parameter $\rho$ (Figure b) on the diffusion profile at $t=0.1$ with $D=2.5$, $M=40$, $N=M^2$.}
\label{fig5}
\end{figure}
Next, we compare the convergence of the two numerical schemes. Table \ref{tab5} shows the maximum error and convergence order of the two numerical schemes for different $\alpha $ and $N$ at $t=1$, where the parameters used are $M=2^{11}$, $D=2$, $\rho =0.5$, $r=\frac{2(2-\alpha )}{\alpha }$, $\epsilon =10^{-9}$. It can be seen that the optimal $2-\alpha $ order is obtained, which illustrates the effectiveness of the L1 scheme on the graded mesh to solve the tempered problem. Table \ref{tab6} presents the error and CPU time comparison between the normal numerical scheme (\ref{eq4.12}) and the fast numerical scheme (\ref{eq4.15}) for different $\alpha $ and $M$ with $N=M^{2}$ at $t=1$, where the parameters utilised are $D=2$, $\rho =0.5$, $r=\frac{2(2-\alpha )}{\alpha }$, $\epsilon =10^{-9}$. We can observe that the L1 scheme on the graded mesh is time-consuming for a large $M$ and the
amount of time for $\alpha =0.4$ and $\alpha =0.8$ is not too much different. In contrast to the normal numerical scheme, the fast numerical scheme can reduce the CPU time significantly without losing accuracy, and there is a clear difference for the total time between the case $\alpha =0.4$ and $\alpha =0.8$. This demonstrates the strong feasibility and applicability to adapt the fast method to deal with the tempered problem. Furthermore, we analyse the impacts of the fractional index $\alpha $ and tempered parameter $\rho $. Figure \ref{fig5a} shows that for the general anomalous diffusion $\rho =0$, the fractional order can boost the diffusion compared to the classical diffusion ($\alpha =1$). The smaller the fractional order is, the faster it diffuses. Compared to the fractional order $\alpha $, the tempered parameter $\rho $ (Figure \ref{fig5b}) can accelerates the diffusion further. The larger $\rho $ is, the more rapid it decays, and with a large $\alpha $ and a moderate $\rho $ ($\alpha =0.9$, $\rho =3.0$) it diffuses faster than with a single fractional order $\alpha $ ($\alpha =0.6$), which affirms the discussion aforementioned again.

%

\section{Application III: Tempered two-layered problems} \label{Sec5}

\begin{figure}[H]
\begin{center}
\scalebox{0.5}[0.5]{\includegraphics{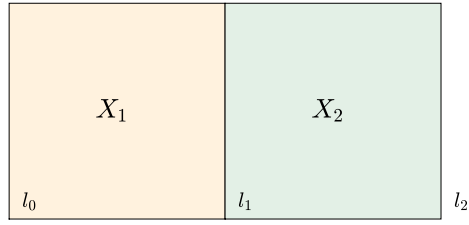}}
\caption{An illustration of a two-layered problem.}
\label{fig6}
\end{center}
\end{figure}
It is well known that the diffusion-induced MRI signal attenuation curve diverges from the mono-exponential decay at high b-values for human brain tissues. A stretched exponential model \cite{Bennett03} was proposed to describe the diffusion-induced signal attenuation effectively, which proves to be a fundamental extension of the fractional Bloch-Torrey equation \cite{Magin08}. In \cite{Zhou10}, Zhou and co-workers applied the fractional models to analyse the diffusion images of human brain tissues in vivo, in which there was a clear contrast between the gray matter and white matter for the diffusivity ($D$) and fractional index ($\beta $). For example, the diffusivity and fractional index for the white matter are $D\approx (0.41\pm0.008)\times 10^{-3}$ and $\beta \approx 0.64\pm 0.01$ while those for the grey matter are $D\approx (0.66\pm 0.007)\times 10^{-3}$ and $\beta \approx0.82\pm 0.01$. It means different tissues exhibit distinct memory or heterogeneity in a heterogeneous medium, which motives us to generalise the
fractional-order model to consider the anomalous diffusion in a composite material. To date, Zeng et al. \cite{Zeng19} used a discrete least squares collocation method to deal with a coupled system of time-fractional diffusion equations with different fractional indices in an irregularly shaped region. Feng et al. \cite{Feng21,Feng21b} proposed a 2D
time-fractional model to study the underlying transport phenomena in a binary medium based on the Riemann-Liouville fractional derivative, in which it showed that the generalised transport model could exhibit the correct physical solution behaviour and produce a more accurate overall mass balance. In this section, we focus on a two-layered problem with a tempered operator (see Figure \ref{fig6}):
\begin{align} \label{eq5.1}
{^{C}_0D_t^{(\alpha_1,\rho_1)}}X_1(x,t)&=D_1\frac{\partial^2X_1(x,t)}{\partial x^2}+S_{a_1}+S_{b_1}X_1,\quad x\in(l_0,l_1),\\\label{eq5.2}
{^{C}_0D_t^{(\alpha_2,\rho_2)}}X_2(x,t)&=D_2\frac{\partial^2X_2(x,t)}{\partial x^2}+S_{a_2}+S_{b_2}X_2,\quad x\in(l_1,l_2),
\end{align}
subject to the initial and boundary conditions
\begin{align} \label{eq5.3}
&X_1(x,0)=X_{1,0}(x),\quad X_2(x,0)=X_{2,0}(x),\\ \label{eq5.4}
&X_1(l_0,t) =f_L(t),\quad  X_2(l_2,t)=f_R(t),
\end{align}
where $D_i$ is the diffusivity coefficients, $S_{a_i}$ and $S_{b_i}$ are some constants. To guarantee the flux exchange at the interface $x=l_1$ is consistent with intrinsic physics, we define the following boundary conditions
\begin{align} \label{eq5.5}
X_1(l_1,t)=X_2(l_1,t),\quad D_{1}\frac{\partial X_1(l_1,t) }{\partial x}=D_{2}\frac{\partial X_2(l_1,t) }{\partial x}.
\end{align}

\subsection{Semi-analytical solution}

Problem (\ref{eq5.1})-(\ref{eq5.5}) will be solved using the finite Fourier and Laplace transforms. Define $d_i=l_i-l_{i-1},i=1,2$, $\lambda_{i,n}$ be the eigenvalues and $\varphi_{i,n}(x)$ be the corresponding eigenfunctions. Then the Sturm-Liouville system in each layer reads:
\begin{align*}
-\frac{d^2  \varphi_{1,n}(x)}{dx^2} = \lambda_{1,n}^2 \varphi_{1,n}(x), \quad\varphi_{1,n}(l_{0})=0,\quad \frac{\varphi_{1,n}(l_{1})}{dx}=0,\\
-\frac{d^2 \varphi_{2,n}(x)}{dx^2}  = \lambda_{2,n}^2 \varphi_{2,n}(x), \quad \frac{\varphi_{2,n}(l_{1})}{dx}=0,\quad\varphi_{2,n}(l_{2})=0.
\end{align*}
It is straightforward to derive that $\lambda_{1,n}=\frac{(2n+1)\pi}{2d_1}$, $\varphi_{1,n}=\sqrt{\frac{2}{d_1}}\sin[\lambda_{1,n}(x-l_{0})]$ and $\lambda_{2,n}=\frac{(2n+1)\pi}{2d_2}$, $\varphi_{2,n}=\sqrt{\frac{2}{d_2}}\sin[\lambda_{2,n}(l_{2}-x)]$. Define the finite Fourier transform $\widetilde{X}_i(\lambda_{i,n},t) := \langle X_i,\varphi_{i,n} \rangle= \int_{l_{i-1}}^{l_i} \, X_i(x,t)\, \varphi_{i,n}(x)\, dx$ within the $i^{\mathrm{th}}$ layer. We apply the finite Fourier transform to obtain
\begin{align}\label{eq5.7}
{^{C}_0D_t^{(\alpha_i,\rho_i)}} \widetilde{X}_i &=D_{i}  \left\langle   \frac{\partial^2 X_i}{\partial x^2}, \varphi_{i,n}  \right\rangle + S_{a_i} \left\langle 1,\varphi_{i,n}  \right\rangle + S_{b_i} \,\widetilde{{X}_i},\quad i=1,2.
\end{align}
Integrating by parts the first term on the righthand side yields
\begin{align}
D_{i} \left\langle \frac{\partial^2 X_i}{\partial x^2}, \varphi_{i,n} \right\rangle &= -D_{i} \left\{ \left\langle -\frac{d^2}{dx^2} \varphi_{i,n}, X_i \right\rangle + \left[ X_i(l_{i},t)\frac{d \varphi_{i,n}}{dx}(\l_{i}) - X_i(l_{i-1},t) \frac{d \varphi_{i,n}}{dx}(l_{i-1})\right.\right. \nonumber \\\label{eq5.8}
& - \left. \left. \frac{\partial X_i}{\partial x}(l_i,t) \varphi_{i,n}(l_i) + \frac{\partial X_i}{\partial x}(l_{i-1},t) \varphi_{i,n}(l_{i-1})\right] \right\}.
\end{align}
Next, define $D_{1}\frac{\partial X_1(l_1,t) }{\partial x}=D_{2}\frac{\partial X_2(l_1,t) }{\partial x}:= v_{12}(t)$. Substituting (\ref{eq5.4}), (\ref{eq5.5}) and (\ref{eq5.8}) into (\ref{eq5.7}), the transformed layer equations are then given by
\begin{align*}
{^{C}_0D_t^{(\alpha_1,\rho_1)}} \widetilde{X}_1 &=-D_{1} \lambda_{1,n}^2 \widetilde{X}_1 + D_1f_L(t)\frac{d \varphi_{1,n}(l_0)}{dx} +v_{12}(t)\varphi_{1,n}(l_1)+S_{a_1}\langle 1,\varphi_{1,n}\rangle+S_{b_1}\, \widetilde{X}_1,\\ 
{^{C}_0D_t^{(\alpha_2,\rho_2)}} \widetilde{X}_2 &=-D_{2} \lambda_{2,n}^2 \widetilde{X}_2 - v_{12}(t) \varphi_{2,n}(l_1)-D_2f_R(t)\frac{d \varphi_{2,n}(l_2)}{dx}+ S_{a_2} \langle 1,\varphi_{2,n} \rangle + S_{b_2} \, \widetilde{X}_2,
\end{align*}
together with the transformed initial conditions $\widetilde{X}_i(\lambda_{i,n},0) = \langle X_{i,0}(x),\varphi_{i,n} \rangle :=\widetilde{X}_{i,0}, i=1,2$. We now apply the Laplace transform in time and denote $\overline{\widetilde{X}}_i(\lambda_{i,n},s) = \mathcal{L}\left\{ \widetilde{X}_i(\lambda_{i,n},t) \right\},i=1,2$ to obtain
\begin{align*}
(s+\rho_1)^{\alpha_1} \overline{\widetilde{X}}_1 -(s+\rho_1)^{\alpha_1-1} \widetilde{X}_{1,0}  = -D_{1} \lambda_{1,n}^2  \overline{\widetilde{X}}_1 + D_1\bar{f_L}(s)\frac{d \varphi_{1,n}(l_0)}{dx} +\bar{v}_{12}(s) \varphi_{1,n}(l_1)+\frac{S_{a_1}{\mathbb{I}_{1,n}}}{s}  + S_{b_1} \overline{\widetilde{X}}_1,     \\
(s+\rho_2)^{\alpha_2} \overline{\widetilde{X}}_2 -(s+\rho_2)^{\alpha_2-1}  \widetilde{X}_{2,0} = -D_{2} \lambda_{2,n}^2  \overline{\widetilde{X}}_2-\bar{v}_{12}(s) \varphi_{2,n}(l_1)-D_2\bar{f}_{R}(s)\frac{d \varphi_{2,n}(l_2)}{dx}+\frac{S_{a_2}{\mathbb{I}_{2,n}}}{s}  + S_{b_2} \overline{\widetilde{X}}_2,
\end{align*}
which can be rearranged into the form
\begin{align}\label{eq5.11}
\overline{\widetilde{X}}_1(\lambda_{1,n},s)=\frac{(s+\rho_1)^{\alpha_1-1}\widetilde{X}_{1,0}}{\eta_{1,n}(s)} +\frac{1}{\eta_{1,n}(s)}\left[D_1\bar{f_L}(s)\frac{d \varphi_{1,n}(l_0)}{dx} +\bar{v}_{12}(s) \varphi_{1,n}(l_1)\right] +\frac{S_{a_1}\mathbb{I}_{1,n}}{s\eta_{1,n}(s)}, \\
\label{eq5.12}
\overline{\widetilde{X}}_2(\lambda_{2,n},s)=\frac{(s+\rho_2)^{\alpha_2-1}\widetilde{X}_{2,0}}{\eta_{2,n}(s)} -\frac{1}{\eta_{2,n}(s)} \left[\bar{v}_{12}(s)\varphi_{2,n}(l_1)+D_2\bar{f}_{R}(s)\frac{d \varphi_{2,n}(l_2)}{dx}\right]+\frac{S_{a_2}\mathbb{I}_{2,n}}{s\eta_{2,n}(s)},
\end{align}
where $\mathbb{I}_{i,n}=\langle 1,\varphi_{i,n} \rangle$ and $\eta_{i,n}(s)=(s+\rho_i)^{\alpha_i}+D_{i} \lambda_{i,n}^{2}-S_{b_i}$, $i=1,2$. In order to calculate the unknown interfacial flux value $\bar{v}_{12}(s)$, we need the boundary condition at the interface:
\begin{align*}
\overline{X}_1(l_1,s) =\overline{X}_2(l_1,s).
\end{align*}
Noting that  $\overline{X}_i(x,s)=\sum_{n=0}^{\infty}\overline{\widetilde{X}}_i(\lambda_{i,n},s) \varphi_{i,n}(x)$, $i=1,2$, then we have
\begin{align}\label{eq5.14}
\sum_{n=0}^{\infty}\overline{\widetilde{X}}_1(\lambda_{1,n},s) \varphi_{1,n}(l_1)= \sum_{n=0}^{\infty}\overline{\widetilde{X}}_2(\lambda_{2,n},s) \varphi_{2,n}(l_1).
\end{align}
Substituting the expressions (\ref{eq5.11}) and (\ref{eq5.12}) into (\ref{eq5.14}), (\ref{eq5.14}) can be rearranged into the following form:
\begin{align*}
& \sum_{n=0}^\infty \left\{ \frac{\varphi^2_{1,n}(l_1)}{\eta_{1,n}(s)}  +   \frac{\varphi^2_{2,n}(l_1)}{\eta_{2,n}(s)} \right\} \overline{v}_{12}(s)\nonumber\\
=&-\sum_{n=0}^\infty \left(\frac{(s+\rho_1)^{\alpha_1-1}\widetilde{X}_{1,0}}{\eta_{1,n}(s)}+\frac{D_1\bar{f_L}(s)}{\eta_{1,n}(s)}\frac{d \varphi_{1,n}(l_0)}{dx}+\frac{S_{a_1}\mathbb{I}_{1,n}}{s\eta_{1,n}}\right) \varphi_{1,n}(l_1)\nonumber\\
+&\sum_{n=0}^\infty \left( \frac{(s+\rho_2)^{\alpha_2-1}\widetilde{X}_{2,0}}{\eta_{2,n}(s)} -\frac{D_2\bar{f}_{R}(s)}{\eta_{2,n}(s)}\frac{d \varphi_{2,n}(l_2)}{dx}+\frac{S_{a_2}\mathbb{I}_{2,n}}{s\eta_{2,n}(s)}
  \right) \varphi_{2,n}(l_1),
\end{align*}
which can be solved for $\overline{v}_{12}(s)$ at a given value of $s$. Applying the inverse Laplace transform and solved the solutions numerically within each layer, we can obtain
\begin{align*}
{\widetilde{X}}_i(\lambda_{i,n},t)&=\mathcal{L}^{-1}\left\{ {\overline{\widetilde{X}}}_i(\lambda_{i,n},s) \right\} = \frac{1}{2 \pi i} \int_{\Gamma} e^{st} \overline{{\widetilde{X}}}_i(\lambda_{i,n},s) \, ds\\
&= \frac{1}{2 \pi i} \int_{\Gamma} \frac{e^z}{t}  \overline{{\widetilde{X}}}_i(\lambda_{i,n},z/t) \, dz \approx -2\Re\left( \sum_{k=1}^{K/2}c_{2k-1} \frac{ \overline{{\widetilde{X}}}_i(\lambda_{i,n},z_{2k-1}/t)}{t}\right),
\end{align*}
where $z=st$, $c_{2k-1}$ and $z_{2k-1}$ are the residues and poles of the near-best minimax approximation of $e^z$ on the negative real line by rational functions of type $(K,K)$ as computed by the Carath\'{e}odor--Fej\'{e}r method \cite{Tref06}. The final solution in each layer can be written as
\begin{equation*}
X_i(x,t)=\sum_{n=0}^{\infty}\widetilde{X}(\lambda_{i,n},t) \varphi_{i,n}(x),\quad i=1,2.
\end{equation*}

\subsection{Numerical solution }

Firstly, we do the grid partition. Define $t_n=T\left( \frac{n}{N} \right)^r $, $r=\min\left\{ \frac{2(2-\alpha_1)}{\alpha_1} ,\frac{2(2-\alpha_2)}{\alpha_2}  \right\}$, $n=0,1,2,\ldots,N$, $\tau_n=t_n-t_{n-1}$, $n=1,2,\ldots,N$. Define $x_i=l_0+ih$, $i=0,1,2,\ldots,M$, where $h=\frac{l_2-l_0}{M}$ is the uniform spatial step. Divide the grid points into two parts: on layer 1 $\{ x_0,x_1,\ldots,x_{M_1}\}$ and on layer 2 $\{ x_{M_1}, x_{M_1+1}, \ldots,x_{M}\}$. Applying the L1 scheme and central difference scheme, we can obtain the numerical solution of the two-layered problem:
\begin{align*} 
&{_0^C\mathbb{D}_t^{(\alpha_1,\rho_1)}}X^n_{1,i}=D_1\delta_x^2X^n_{1,i}+S_{a_1}+S_{b_1}X^n_{1,i},\quad 1\leq i\leq M_1,~1\leq n\leq N,\\
&{_0^C\mathbb{D}_t^{(\alpha_2,\rho_2)}}X^n_{2,i}=D_2\delta_x^2X^n_{2,i}+S_{a_2}+S_{b_2}X^n_{2,i},\quad M_1\leq i\leq M,~1\leq n\leq N.
\end{align*}
Exploiting the boundary conditions (\ref{eq5.5}) to overlap the value in $x=l_1$, the matrix form in terms of the solution on $[l_0,l_2]$ can be derived, which can be solved using the general iterative method.

\subsection{Numerical examples}

\begin{figure}[H]
\begin{center}
\scalebox{0.5}[0.5]{\includegraphics{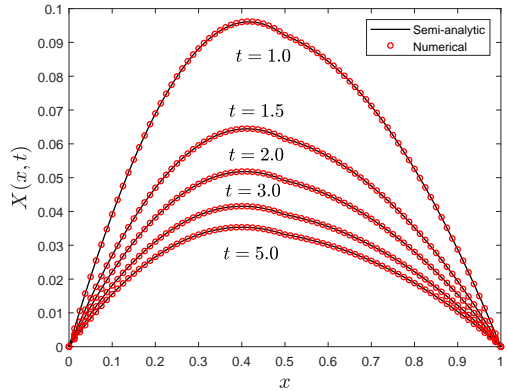}}
\caption{The comparison between the numerical solution and the semi-analytical solution at different time for problem (\ref{eq5.1})-(\ref{eq5.5}), where the parameters chosen are $\alpha_1=0.9$, $\alpha_2=0.8$, $\rho_1=0.1$, $\rho_2=0.5$, $D_1=0.25$, $D_2=0.5$, $S_{a_1}=S_{a_2}=0.1$, $S_{b_1}=S_{b_2}=-0.1$.}
\label{fig7}
\end{center}
\end{figure}

We consider the two-layered problem on a domain $[0,1]$ with $l_{1}=0.5$ and initial conditions $X_{1,0}(x)=X_{2,0}(x)=1$ and boundary conditions $f_{L}(t)=f_{R}(t)=0$. Figure \ref{fig7} shows a comparison between the semi-analytical solution and the numerical solution at different times, from which there is perfect agreement between the semi-analytical solution and
the numerical solution. It demonstrates that both methods work well. Next, we consider a scenario where one layer is normal material characterised by $\alpha _{1}=1$, $\rho _{1}=0$ while the other layer exhibits memory ($\alpha_{2}\neq 1$). Figure \ref{fig8a} presents the effects of the fractional index $\alpha _{2}$ on the solution profile. It can be seen that a small $
\alpha _{2}$ can accelerate the diffusion. Figure \ref{fig8b} displays the impacts of the tempered parameter $\rho _{2}$ on the solution profile. Similarly, the tempered parameter $\rho _{2}$ can further promote the decay and with a moderate parameter pair $\alpha _{2}=0.9$, $\rho _{2}=0.5$, it decays faster than that only with a small $\alpha _{2}$. We can see that the influence of the tempered parameter on the two-layered problem is similar to that on a homogeneous medium.

\begin{figure}[H]
\centering
\subfloat[Varying $\alpha$  ]{
\label{fig8a}
\begin{minipage}[t]{0.46\textwidth}
\centering
\scalebox{0.45}[0.45]{\includegraphics{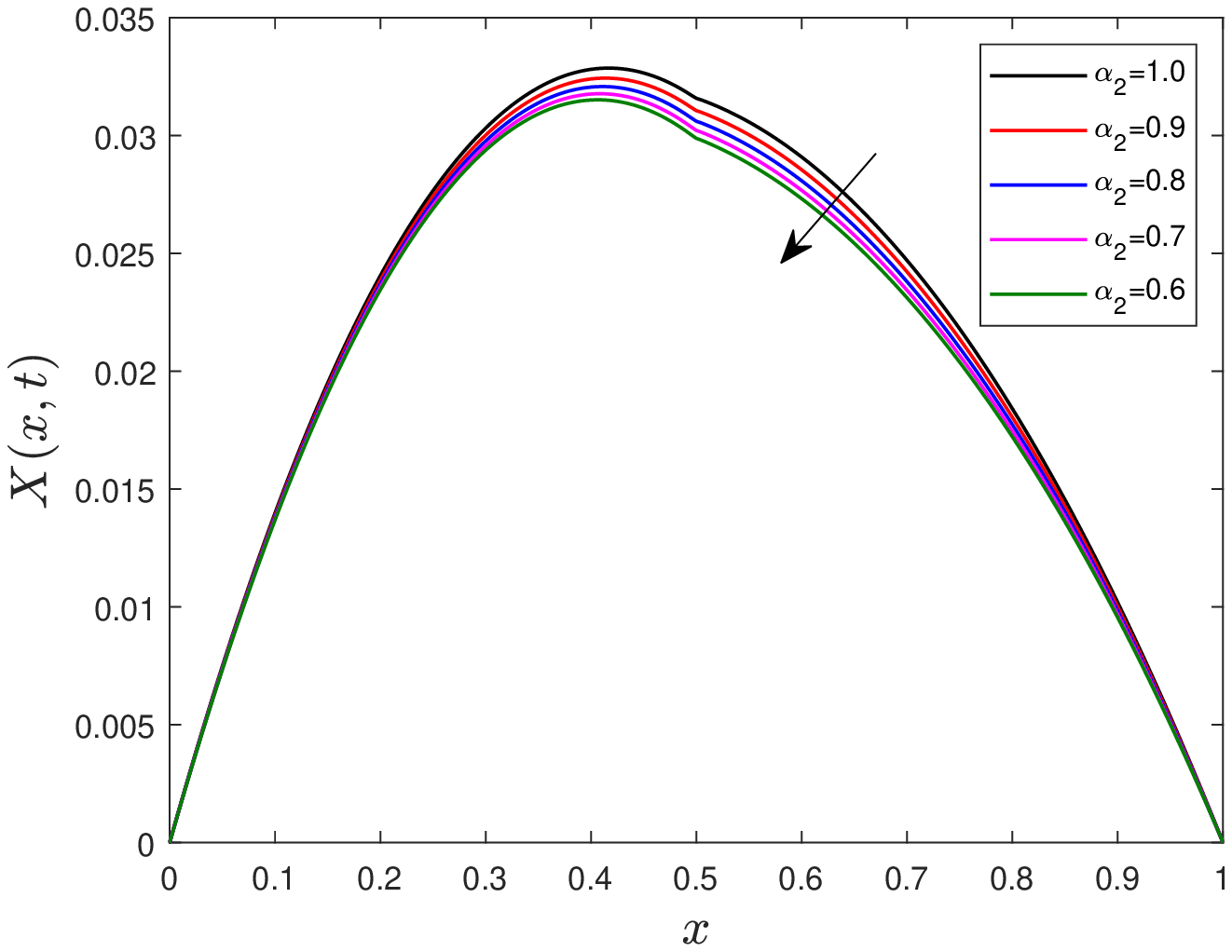}}
\end{minipage}
}
\subfloat[Varying $\rho$ ]{
\label{fig8b}
\begin{minipage}[t]{0.46\textwidth}
\centering
\scalebox{0.45}[0.45]{\includegraphics{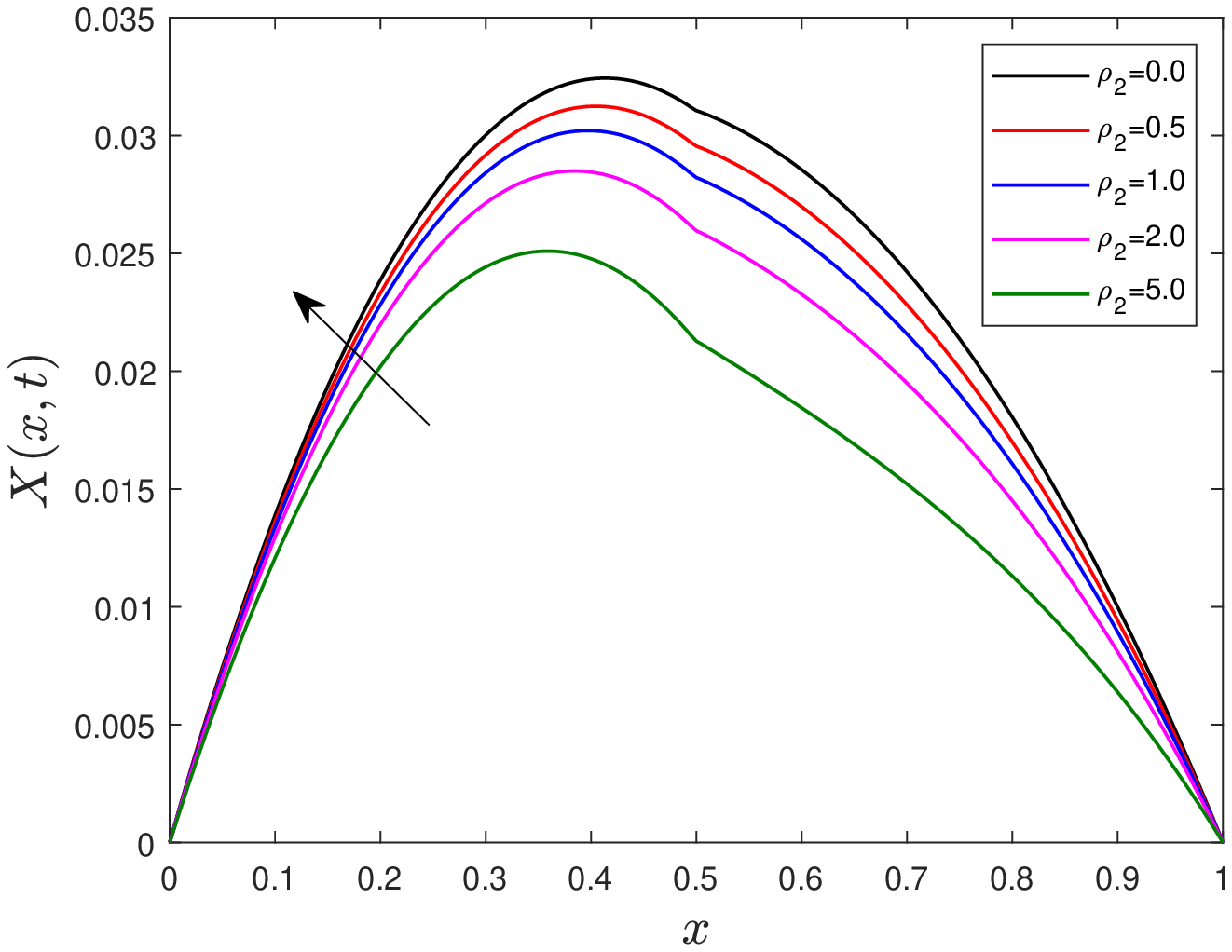}}
\end{minipage}
}
\caption{The impacts of the fractional index $\alpha$ (Figure a) and tempered parameter $\rho$ (Figure b) on the diffusion profile at $t=0.1$ with $D_1=0.25$, $D_2=0.5$, $S_{a_1}=S_{a_2}=0.1$, $S_{b_1}=S_{b_2}=-0.1$, (a) $\alpha_1=1$, $\rho_1=0$, $\rho_2=0.0$, (b) $\alpha_1=1$, $\rho_1=0$, $\alpha_2=0.9$.}
\label{fig8}
\end{figure}


\section{Conclusions}\label{Sec6}

In this paper, we extend two classical numerical schemes, the L1 scheme on graded mesh and the WSGL formula with correction terms, to deal with the benchmark problem with a tempered operator. Both schemes are effective. In addition, a fast algorithm for the time tempered Caputo derivative is developed to reduce the running time significantly. Furthermore, the tempered operator is applied to different models to investigate the tempered solution behaviour. An important finding is that, compared with the fractional index, the tempered parameter could further accelerate the diffusion, and the tempered model with two parameters $\alpha $ and $\rho $
is more flexible. In the future, we will explore high-dimensional tempered diffusion problems in heterogeneous media.

\section*{Acknowledgements} \label{Sec7}

This research was supported by the Australian Research Council via the Discovery Projects DP180103858 and DP190101889 and National Natural Science Foundation of China (No. 11801543).

\end{document}